\newif\ifHAL
\HALtrue

\ifHAL
\documentclass[10pt,a4paper]{article}
\usepackage{a4wide,geometry}
\else
\documentclass[10pt]{siamltex}
\fi

\usepackage[utf8]{inputenc}
\usepackage[T1]{fontenc}
\usepackage[english]{babel}
\usepackage{color,graphicx}

\usepackage{amsmath,amssymb}

\ifHAL
\usepackage{amsthm}
\newtheorem{theorem}{Theorem}[section]
\newtheorem{assumption}[theorem]{Assumption}
\newtheorem{proposition}[theorem]{Proposition}
\newtheorem{lemma}[theorem]{Lemma}

\newtheorem{remark}[theorem]{Remark}
\newtheorem{definition}[theorem]{Definition}
\else
\newtheorem{remark}[theorem]{Remark}
\newtheorem{assumption}[theorem]{Assumption}
\usepackage{fancyhdr}

\fi

\numberwithin{equation}{section}

\usepackage{xcolor}
\usepackage{scalerel}

\input mydef.sty

\begin{document}

\ifHAL

\title{$\Hrt$-reconstruction of piecewise polynomial fields with application to $hp$-a posteriori nonconforming error analysis for Maxwell's equations}

\author{Z. Dong\footnotemark[1], A. Ern\footnotemark[2]}

\footnotetext[1]{Inria, 48 rue Barrault, 75647 Paris, France and CERMICS, ENPC, Institut Polytechnique de Paris, CNRS, 6 \& 8 avenue B.~Pascal, 77455 Marne-la-Vall\'{e}e, France}

\footnotetext[2]{CERMICS, ENPC, Institut Polytechnique de Paris, CNRS, 6 \& 8 avenue B.~Pascal, 77455 Marne-la-Vall\'{e}e, France and Inria, 48 rue Barrault, 75647 Paris, France}

\date{\today}

\fi

\maketitle

\begin{abstract}
We devise and analyze a novel $\Hrt$-reconstruction operator for piecewise polynomial fields on shape-regular simplicial meshes. The (non-polynomial) reconstruction is devised over the mesh vertex patches using the partition of unity induced by hat basis functions in combination with local Helmholtz decompositions. Our main focus is on homogeneous tangential boundary conditions. We prove that the difference between the reconstructed $\Hrtz$-field and the original, piecewise polynomial field, measured in the broken curl norm and in the $\bL^2$-norm, can be bounded in terms of suitable jump norms of the original field. The bounds are always $h$-optimal, and $p$-suboptimal by $\frac12$-order for the broken curl norm and by $\frac32$-order for the $\bL^2$-norm. An auxiliary result of independent interest is a novel broken-curl, divergence-preserving Poincar\'e inequality on vertex patches. Moreover, the $\bL^2$-norm estimate can be improved to $\frac12$-order suboptimality under a (reasonable) assumption on the uniform elliptic regularity pickup for a Poisson problem with Neumann conditions over the vertex patches.
We also discuss extensions of the $\Hrtz$-reconstruction operator to the prescription of mixed boundary conditions, to agglomerated polytopal meshes, and to convex domains.
Finally, we showcase an important application of the $\Hrt$-reconstruction operator to the $hp$-a posteriori nonconforming error analysis of Maxwell's equations. We focus on the (symmetric) interior penalty discontinuous Galerkin (dG) approximation of some simplified forms of Maxwell's equations.
\end{abstract}

\ifHAL



\fi

\section{Introduction}

Given a piecewise polynomial field on a shape-regular simplicial mesh covering
exactly a Lipschitz polyhedral domain, we show in
the paper the existence of a (non-polynomial) $\Hrt$-conforming field whose distance
to the original field in the broken curl norm and in the $\bL^2$-norm is bounded in terms
of suitable jumps (and boundary values) of the original piecewise polynomial field,
namely its tangential jumps (and boundary values) and the normal jumps (and boundary values)
of its curl. All the estimates are local and
$h$-optimal, and they are
$p$-suboptimal by only $\frac12$-order in the broken curl norm and by $\frac32$-order
in the $\bL^2$-norm. Here, $h$ denotes the mesh size and $p$ the polynomial degree.
We emphasize that, as an auxiliary result to establish the $\bL^2$-norm estimate,
we derive a novel broken-curl, divergence-preserving Poincar\'e inequality on
the mesh vertex patches that is of independent interest.
We also observe that the $\bL^2$-estimate can be improved to $\frac12$-order
$p$-suboptimality
under a (reasonable) assumption on the uniform elliptic regularity pickup for a Poisson
problem with Neumann conditions over the mesh vertex patches.

We present an explicit construction of the (non-polynomial) $\Hrt$-conforming field over the mesh vertex patches using the partition of unity induced by hat basis functions in combination with local Helmholtz decompositions. The crucial advantage of resorting to a local construction over the mesh vertex patches is that the associated subdomains have a trivial topology, thereby circumventing any dependence of the estimates on the topology of the  original domain. For simplicity, we devise the construction so that the (non-polynomial)  $\Hrt$-conforming field satisfies homogeneous tangential boundary conditions,  i.e., sits in $\Hrotz$, but we also  discuss the extension to more general mixed boundary conditions. Another relevant extension is the case of agglomerated polytopal meshes, where the same  $hp$-estimates as on simplicial meshes can be established by considering the underlying  simplicial submesh under the assumption that this latter mesh has locally the same size as the original polytopal mesh. Finally, we also discuss a variant where the  $\Hrt$-conforming field is defined globally without considering the mesh vertex patches. On convex domains, this leads to a $p$-optimal $\bL^2$-estimate  without any further assumption, but the estimate is only global and the constant depends on the topology of the global domain.

Reconstructing $\Hrt$-conforming fields from piecewise polynomial fields on a shape-regular simplicial mesh has already been quite explored in the literature. A first possibility is to devise a polynomial $\Hrt$-reconstruction using N\'ed\'elec (edge) finite elements defined by prescribing the degrees of freedom as the averages of those of the original, piecewise polynomial field. Such averaging operators have been considered in \cite{CamSo:16}, and a systematic analysis can be found in \cite{ErnGu:17_quasi} with and without boundary prescription, leading to (local) $h$-optimal estimates in the broken curl and $\bL^2$-norms (see also \cite{Brenn:93,Oswal:93} for the seminal work in the $H^1$-context). However, defining a reconstruction by prescription of degrees of freedom can hardly lead to $p$-optimal estimates because of the repeated use of discrete trace (and inverse) inequalities needed in the analysis (see \cite{BurEr:07} for further insight in the $H^1$-context). A second possibility is to solve local, discrete minimization problems using again N\'ed\'elec and Raviart--Thomas finite elements. Following the seminal work in \cite{BraSc:08,ErnVo:15,ErnVo:20} in the $H^1$/$\Hdv$-context, recent advances leading, in particular, to full $p$-optimality in the $\Hrt$-context were accomplished in \cite{ChaEV:21,ChaVo:23}. \cite{ChaEV:21} poses local problems on edge patches, but does not achieve $\Hrt$-conformity. Instead, \cite{ChaVo:23} poses local problems on vertex patches and achieves $\Hrt$-conformity. However, in both works, the original piecewise polynomial field must satisfy a certain orthogonality property, typically associated with Galerkin's orthogonality when solving a curl-curl problem. We also refer the reader to \cite{Chaum:23} for a simplified construction of \cite{ChaVo:23}, and to \cite{ChaVo:24} for the use of the technique from \cite{ChaVo:23} to devise a stable, local, commuting operator in $\Hrt$ with $hp$-optimal approximation properties. The above discussion shows that a gap remains in the literature when it comes to devising an $\Hrt$-conforming reconstruction from an arbitrary piecewise polynomial field and ensuring $hp$-approximation properties.

The main application we have in mind for the present $\Hrt$-reconstruction operator is the $hp$-a posteriori nonconforming error analysis of some simplified forms of Maxwell's equations involving the curl-curl operator and a zero-order term. In the case of elliptic problems, the idea of invoking an $H^1$-reconstruction operator based on a global Helmholtz decomposition to estimate the so-called nonconforming approximation error can be traced back to \cite{DarDurPraVam1996} in 2D and to \cite{CarBarJan2002} in 3D with Dirichlet boundary conditions (see also \cite{BecHanLar2003,Ainsworth:07} where the Helmholtz decomposition is also invoked). A further advance has been accomplished recently in \cite{DongErn:24} in the context of a discretization by the hybrid high-order (HHO) method introduced in \cite{DiPEL:14} for elliptic problems including mixed Dirichlet--Neumann boundary conditions. Specifically, the $H^1$-reconstruction operator considered in \cite{DongErn:24} is devised locally on mesh vertex patches which all have a simple topology.

We consider the (symmetric) interior penalty discontinuous Galerkin (dG) method for space discretization. This approximation technique is classical for elliptic problems (see, e.g., \cite{ArBCM:01,DiPEr:12} and the references therein) and has been developed and analyzed in the context of Maxwell's equations in \cite{PeScM:02,HoPSS:05}.
Here, we perform the $hp$-a posteriori error analysis hinging on the present $\Hrt$-reconstruction operator, focusing on homogeneous Dirichlet boundary conditions. This analysis offers, to our knowledge, two ground-breaking features: it is the first time that a full $hp$-analysis is achieved, and it is the first time that the analysis hinges on local (vertex patch) Helmholtz decompositions to estimate the nonconforming error. We emphasize that, while the theoretical derivation of the error estimate relies on the $\Hrtz$-reconstruction operator, this reconstruction does not need to be computed in practice.
Previous a posteriori error estimates bounded the nonconforming error using an averaging operator, as in \cite{houstonpersch04,HoPeS:07,chaumontfrelet:hal-04589791}, and were therefore only $h$-optimal. We notice, in passing, that \cite{chaumontfrelet:hal-04589791} derived an $hp$-a posteriori bound on the conforming error, and addressed the frequency-dependence of the a posteriori estimates in the indefinite regime by using a dualtiy argument \`a la Schatz (see \cite{Schatz1974}).

The paper is organized as follows. In Section~\ref{sec:setting}, we briefly present the continuous and discrete settings. In Section~\ref{sec:main results}, we state our main results on the $\Hrtz$-reconstruction operator and comment on various possible extensions. The main result of this section is Theorem~\ref{theorem: nonconforming error total}.
In Section~\ref{sec:preparatory}, we establish some preparatory results. In particular, we highlight the local broken-curl, divergence-preserving Poincar\'e inequality from Lemma~\ref{lemma: broken_Poincare inequality}.
In Section~\ref{sec:proofs}, we prove our main result. Finally, in Section~\ref{sec: a posteriori error for dG}, we consider the approximation of Maxwell's equations by the (symmetric) interior penalty dG method and develop an $hp$-version a posteriori error analysis based on the $\Hrtz$-reconstruction operator devised in Section~\ref{sec:main results}. The main result of this section is Theorem~\ref{th:reliability}.

\section{Setting}
\label{sec:setting}

In this section, we briefly describe the continuous and discrete settings.

\subsection{Continuous setting}

Let $\Dom$ be an open, bounded, connected, Lipschitz set (domain) in $\Real^d$ with boundary $\front$ and outward unit normal $\bn_\Dom$. We assume that $\Dom$ is a polyhedron so that it can be covered exactly by a simplicial mesh. To fix the ideas, we assume that $d=3$, but all what follows readily extends to $d=2$. We do not make any simplifying assumption on the topology of $\Dom$, i.e., $\Dom$ may be multiply connected and its boundary be composed of several connected components.

We use boldface font for vectors, vector fields, and functional spaces composed of such fields.
We use standard notation for Lebesgue and Sobolev spaces, and fractional-order Sobolev spaces are equipped with the standard Sobolev--Slobodeckii seminorm based on the double integral
(see, e.g., \cite[Def.~2.14]{EG_volI}). To alleviate the notation,
the inner product and associated norm in the spaces $\Ldeux$ and $\Ldeuxd$
are denoted by $(\SCAL,\SCAL)$ and $\|\SCAL\|$, respectively, whereas we add a subscript
$\omega$ when we consider the Lebesgue spaces over a subdomain $\omega \subset \Dom$.
We consider the Hilbert Sobolev spaces
\begin{subequations} \label{eq:Hrot_spaces} \begin{align}
\hrot & \eqq \{\bv \in \bL^2(\omega)  \st \ROT\bv\in \bL^2(\omega)\},\\
\hrotz &\eqq \{\bv \in \hrot \st \gamma\upc_{\partial\omega}(\bv)=\bzero\}, \\
\hdiv & \eqq \{\bv \in \bL^2(\omega)  \st \DIV\bv\in L^2(\omega)\},\\
\hdivz &\eqq \{\bv \in \hdiv \st \gamma\upd_{\partial\omega}(\bv)=0\},
\end{align} \end{subequations}
where $\gamma\upc_{\partial\omega}:\hrot\rightarrow \bH^{-\frac12}(\partial\omega)$
(resp., $\gamma\upd_{\partial\omega}:\hdiv\rightarrow H^{-\frac12}(\partial\omega)$) is
the extension by density of the tangential (resp., normal) trace operator such
that $\gamma\upc_{\partial\omega}(\bv)=\bv|_{\partial\omega}\CROSS \bn_\omega$ (resp.,
$\gamma\upd_{\partial\omega}(\bv)=\bv|_{\partial\omega}\SCAL \bn_\omega$) for smooth fields.
The subscript ${}_0$ is used for the nabla operator to indicate that the operator acts on
functions or fields respecting suitable homogeneous Dirichlet conditions. Thus,
$(\GRAD,-\DIVZ)$, $(\ROT,\ROTZ)$ and $(\DIV,-\GRADZ)$ are pairs of adjoint operators;
for instance, $(\ROTZ\bv,\bw)=(\bv,\ROT\bw)$ for all
$(\bv,\bw)\in \Hrotz\times\Hrot$.

\subsection{Mesh and polynomial spaces}

Let $\calT_h$ be an affine simplicial mesh covering $\Dom$ exactly. A generic mesh cell is denoted by $K$, its diameter by $h_K$ and its outward unit normal by $\bn_K$.
We write $\Fall$ for the set of mesh faces, $\Fint$
for the subset of mesh interfaces (shared by two distinct mesh
cells, $K_l$, $K_r$), and $\Fb$ for the subset of mesh boundary
faces (shared by one mesh cell, $K_l$, and the boundary, $\front$).
Every mesh interface $F\in\Fint$ is oriented
by the unit normal, $\bn_F$, pointing from $K_l$ to $K_r$
(the orientation is arbitrary, but fixed).
Every boundary face $F\in\Fb$ is oriented by the unit normal $\bn_F\eqq\bn_\Dom|_F$.
The diameter of a generic face $F\in\Fall$ is denoted by $h_F$.

The set of mesh vertices is denoted by $\vertice$ and is decomposed into the subset of interior vertices, $\verticei$, and the subset of boundary vertices, $\verticeb$. For all $\vtx\in \vertice$, $\meshv$ denotes the collection of the mesh cells which share $\vtx$ and $\omvtx$ the corresponding open subdomain (often called vertex patch).
The set of mesh faces belonging to $\overline{\omvtx}$, say $\bFa$, is partitioned as $\bFa=\Fv\cup \Fma$, where
$\Fv$ is the collection of the mesh faces which share $\vtx$ and $\Fma$ the collection of all the mesh faces lying on $\partial\omvtx$ and not containing $\vtx$. We notice that the two sets $\Fv$ and $\Fma$ are disjoint and that the set $\Fv$ contains mesh boundary faces if $\vtx$ lies on the boundary. Furthermore,
for all $\vtx\in\vertice$, we set
$h_\vtx:=\diam(\omvtx)$, and
let $\psi_{\vtx}$ be the hat basis function equal to $1$ at $\vtx$ and
having support in $\omvtx$. Recall that the hat basis functions satisfy
the following partition-of-unity property:
\begin{equation} \label{eq:PU}
\sum_{\vtx\in \vertice} \psi_\vtx = 1.
\end{equation}

For all $K\in \calT_h$, $\calFK$ (resp. $\verticeK$) is the collection of the mesh faces composing $\partial K$ (resp., mesh vertices in $\partial K$).
For all $F\in\calFK$ and every piecewise smooth field $\bv$ on $\mesh$, we define
the local trace operators such that
$\gamma\upg_{K,F}(\bv)(\bx) \eqq \bv|_K(\bx)$,
$\gamma\upc_{K,F}(\bv)(\bx)\eqq \bv|_{K}(\bx)\CROSS\bn_F$, $\gamma\upd_{K,F}(\bv)(\bx)\eqq \bv_h|_{K}(\bx)\cdot \bn_F$, for a.e.~$\bx\in F$.
Then, for all $F\in\Fint$ and $\mathrm{x}\in\{\mathrm{g},\mathrm{c}, \mathrm{d}\}$,
we define the jump and average operators such that
\begin{equation}
\jump{\bv}\upx_F\eqq \gamma\upx_{K_l,F}(\bv)-\gamma\upx_{K_r,F}(\bv), \quad
\avg{\bv}\upx_F\eqq \frac12\big(\gamma\upx_{K_l,F}(\bv)+\gamma\upx_{K_r,F}(\bv)\big).
\end{equation}
We also set $\jump{\bv}\upx_F \eqq \avg{\bv}\upx_F \eqq \gamma\upx_{K_l,F}(\bv)$
for all $F\in\Fb$. For all $K\in\mesh$, we define the jump across the boundary of $K$
as $\jumpK{\bullet}\upx|_F:=\iota_{K,F}\jump{\bullet}\upx_F$ for all $F\in\calFK$,
with $\iota_{K,F}:=\bn_K\SCAL\bn_F=\pm1$.

Let $p\ge1$ be the polynomial degree (see Remark~\ref{rem:p=0} below for the case $p=0$).
Let $\polP_{p,d}$ be the space
composed of $d$-variate polynomials of total degree at most $p$ and set
$\bpolP_{p,d}\eqq [\polP_{p,d}]^d$. We consider the following broken polynomial space:
\begin{equation}
\Ppb \eqq \bset \bv_h\in
\bL^2(\Dom) \st \bv_h|_K\in \bpolP_{p,d},\, \forall K\in\calT_h\eset, \label{eq:def_Vhb}
\end{equation}
and we denote the $\bL^2$-orthogonal projection onto $\Ppb$ as
\begin{equation}\label{def: L2 projection}
\bPi\upb_h : \Ldeuxd \to \Ppb.
\end{equation}
We use the notation $P_p\upb(\meshv)$ for the broken polynomial space based on the local
mesh $\meshv$ for a mesh vertex $\vtx\in\vertice$.

\subsection{$hp$-approximation tools}

Let $\kappa_{\calT_h}$ denote the shape-regularity parameter
of the mesh $\calT_h$, and $C(\kappa_{\calT_h})$ a generic constant
solely depending on $\kappa_{\calT_h}$ (but independent of the mesh size and the
polynomial degree), and whose value can change at each occurrence.
Sometimes, we also abbreviate as $A\lesssim B$ the inequality $A\le C(\kappa_{\calT_h})B$
for positive numbers $A$ and $B$.

\begin{lemma}[Discrete trace inequality]\label{lemma: Inverse inequality}
There exists $C(\kappa_{\mesh})$ so that,
for all $v\in \polP_{p,d}(K)$, all $K\in\mesh$, and all $p\ge1$,
\begin{equation}\label{eq: discrete_trace_inv}
\|v\|_{\dK} \leq C(\kappa_{\mesh}) \Big(\frac{p^2}{h_K} \Big)^{\frac12}\|v\|_{ K}.
\end{equation}
\end{lemma}

\begin{proof}
A proof can be found in \cite{warburton2003constants} (with explicit constant in terms of $d$).
\end{proof}

\begin{lemma}[Local $L^2$-orthogonal projection]\label{lemma: L^2 projection}
Let $\delta \in(0,\frac12]$.
There exists $C(\kappa_{\mesh})$ so that,
for all $v\in H^{\frac12+\delta}(K)$, all $K\in\mesh$, and all $p\geq1$,
\begin{equation}\label{eq: L2 projection Polynomial approximation}
\|{v} - \Pi^p_{ K} (v)\|_{\partial K}
\leq C(\kappa_{\mesh}) \Big(\frac{h_K}{p}\Big)^{\delta} |v|_{H^{\frac12 + \delta}(K)},
\end{equation}
where $\Pi^p_K$ denotes the $L^2$-orthogonal projection onto $\polP_{p,d}$.
\end{lemma}

\begin{proof}
A proof can be found in \cite[Corollary 1.2.]{Melenkurzer14}.
\end{proof}

\begin{lemma}[Local modified $hp$-Karkulik--Melenk operator]\label{lemma:local_hp-KM}
There exists $C(\kappa_{\mesh})$ such that, for all $\vtx \in \vertice$, there is
a local interpolation operator $\Igv: H^1(\omvtx) \to P_p\upb(\meshv) \cap H^1(\omvtx)$
such that, for all $v \in H^1(\omvtx)$, all $K \in \meshv$, and all $p\ge1$,
\begin{equation} \label{eq:inter_Ig on patch}
\Big(\frac{p}{h_K}\Big) \|v-\Igv(v)\|_K
+
\Big(\frac{p}{h_K}\Big)^{\frac12} \|v-\Igv(v)\|_{\partial K}
+
\|\nabla \Igv(v)\|_K
\leq
C(\kappa_{\mesh}) \|\nabla v\|_{\omvtx}.
\end{equation}
\end{lemma}
\begin{proof}
See \cite{Melenk:05,KarMe:15} (see also \cite[Corollary~2.5]{DongErn:24} for
using the $H^1$-seminorm on the right-hand side).
\end{proof}

\subsection{Jump lifting operator}

For every field $\bv_h\in \bP\upb_{p+1}(\calT_h)$, $\ROTh \bv_h$
denotes the broken curl of $\bv_h$ (evaluated cellwise). The reason we consider the
space $\bP\upb_{p+1}(\calT_h)$ is that, in what follows, we shall consider the product of a hat basis function times a field in $\Ppb$.
We define the jump lifting operator $\bLhpz: \bP\upb_{p+1}(\calT_h)\to \bP\upb_p(\calT_h)$
such that
\begin{equation} \label{eq:def_Lh}
(\bLhpz(\bv_h),\bphi_h) \eqq \sum_{F\in\Fall} (\jump{\bv_h}\upc_F,\avg{\bphi_h}\upg_F)_{\bL^2(F)},
\quad \forall \bphi_h\in \bP\upb_p(\calT_h).
\end{equation}

\begin{lemma}[Bound on lifting operator]
There exists $C(\kappa_{\mesh})$ so that,
for all $\bv_h\in \bP\upb_{p+1}(\calT_h)$, all $K\in\mesh$, and all $p\ge1$,
\begin{equation}\label{eq:bnd_on_L}
\|\bLhpz(\bv_h)\|_{K} \leq C(\kappa_{\mesh}) \bigg\{ \sum_{F\in\calFK}
\Big(\frac{p^2}{h_K}\Big) \|\jump{\bv_h}\upc_F\|_{F}^2\bigg\}^{\frac12}.
\end{equation}
\end{lemma}

\begin{proof}
Apply the Cauchy--Schwarz inequality and the discrete trace inequality~\eqref{eq: discrete_trace_inv}.
\end{proof}

\section{Main results on $\Hrt$-reconstruction} \label{sec:main results}

This section presents, and comments on extensions of,
our main result on $\Hrtz$-reconstruction.
Specifically, we construct an operator
\begin{equation}
\opRec : \bP\upb_p(\mesh) \rightarrow \Hrotz,
\end{equation}
such that, for all $\bE_h\in \bP\upb_p(\mesh)$, the error $\bE_h-\opRec(\bE_h)$ is bounded
in the broken curl norm and in the $\bL^2$-norm by suitable jumps of $\bE_h$ and $\ROTh\bE_h$
across the mesh interfaces and by suitable traces on the mesh boundary faces.

The devising of $\opRec(\bE_h)$ combines local (non-polynomial) constructions on the vertex patches associated with each mesh vertex $\vtx\in\vertice$. For all $\vtx\in\vertice$, we define the following
functional spaces:
\begin{subequations} \label{def: Vertex spaces}\begin{align}
\bXzv& := \Hrotzv\cap \Hdivvz, \label{eq:def_bXzv} \\
\bXdivzv& := \Hrotv\cap \Hdivzvz, \label{eq:def_bXdivzv}
\end{align} \end{subequations}
where $\Hdivvz$ (resp., $\Hdivzvz$) is the subspace of $\Hdivv$ (resp., $\Hdivzv$) composed
of divergence-free fields. We also set
\begin{subequations} \label{def: H1_patch}\begin{align}
X\upg_0(\omvtx) &:= H^1_0(\omvtx), \\
X\upg(\omvtx)&:=H^1(\omvtx)/\Real.
\end{align} \end{subequations}

\begin{definition}[Patchwise and global $\Hrtz$-reconstruction]\label{def: Patchwise and global potential reconstruction}
Let $\bE_h\in \bP\upb_p(\mesh)$.
For all $\vtx\in \vertice$, let $\bU_{\vtx}\in \bXzv$ solve the following well-posed problem:
\begin{equation}\label{def: Patch reconstruction}
(\ROTZ \bU_{\vtx}, \ROTZ  \bW_{\vtx})_{\omvtx}= (\bROT  (\psi_{\vtx} \bE_h ) + \bLhpz(\psi_{\vtx} \bE_h), \ROTZ   \bW_{\vtx})_{\omvtx} \quad \forall  \bW_{\vtx} \in  \bXzv,
\end{equation}
and let $\theta_{\vtx}\in X\upg_0(\omvtx)$ solve the following well-posed problem:
\begin{equation}\label{def: Patch reconstruction curl free}
( \nabla \theta_{\vtx}, \nabla v_{\vtx}  )_{\omvtx}= (  \psi_{\vtx} \bE_h  , \nabla v_{\vtx})_{\omvtx}
\quad \forall  v_{\vtx} \in  X\upg_0(\omvtx).
\end{equation}
Set
\begin{equation}\label{def: vertex curl reconstruction}
\bE_{\vtx}:  = \bU_{\vtx}  + \nabla \theta_{\vtx},
\end{equation}
and extending $\bE_{\vtx}$ by zero to $\Dom$,  set
\begin{equation}\label{def: patch reconstruction on the whole domain}
\opRec(\bE_h): = \sum_{\vtx\in \vertice} \bE_{\vtx}.
\end{equation}
\end{definition}

\begin{remark}[Local problems]
As the vertex patch $\omvtx$ is a simply connected Lipschitz domain with connected boundary,
the local problems \eqref{def: Patch reconstruction} and
\eqref{def: Patch reconstruction curl free} are well-posed. Moreover,
an equivalent definition is
\begin{align*}
\bU_{\vtx}&:= \arg \min_{\brho_{\vtx}\in \bXzv}\|  \ROTZ \brho_{\vtx} - \{\bROT(\psi_{\vtx} \bE_h) + \bLhpz(\psi_{\vtx} \bE_h)\}\|_{\omvtx}, \\
\theta_{\vtx}&: = \arg \min_{\xi_{\vtx}   \in X\upg_0(\omvtx) }\|  \nabla  \xi_{\vtx}  - \psi_{\vtx} \bE_h \|_{\omvtx}.
\end{align*}
\end{remark}

A simple but crucial observation is that $\bE_{\vtx}\in \Hrotzv$ so that
\begin{equation}
\opRec(\bE_h) \in \Hrotz.
\end{equation}
We now set
\begin{equation} \label{eq:encvtx}
\bdelta_{\vtx}:= \bE_{\vtx} - \psi_{\vtx} \bE_h \quad \forall \vtx\in\vertice,
\end{equation}
and estimate $\bdelta_{\vtx}$ in the broken curl norm and in the $\bL^2$-norm.

\begin{lemma}[Local bound in broken curl norm]\label{lem:broken curl}
There exists $C(\kappa_{\mesh})$ so that, for all $\bE_h\in \bP\upb_p(\mesh)$ and all $\vtx \in \vertice$,
\begin{align}
\|\bROT \bdelta_{\vtx}\|_{\omvtx} \le {}&
C(\kappa_{\mesh}) \bigg\{
 \! \sum_{F\in \Fv   } \!
 \Big(\frac{h_F}{p}\Big) \ltwo{\jump{ \bROT \bE_h}\upd_F }{F}^2
+ \Big(\frac{p^2}{h_F} \Big) \ltwo{ \jump{\bE_h}\upc_F}{F}^2
\bigg\}^{\frac12}.
\label{eq:broken curl}
\end{align}
\end{lemma}

\begin{proof}
The proof can be found in Section \ref{sec:proof_broken curl}.
\end{proof}

\begin{lemma}[Local $\bL^2$-bound using broken-curl, divergence-preserving Poincaré inequality]\label{lem:L2 Poincare}
There exists $C(\kappa_{\mesh})$ so that, for all $\bE_h\in \bP\upb_p(\mesh)$ and all $\vtx \in \vertice$,
\begin{align}
\| \bdelta_{\vtx} \|_{\omvtx} \le {}&  C(\kappa_{\mesh}) p \bigg\{
 \! \sum_{F\in \Fv   } \!
\Big( \frac{h_F}{p}\Big)^3 \ltwo{\jump{\ROTh \bE_h}\upd_F }{F}^2
+ h_F \ltwo{ \jump{\bE_h}\upc_F}{F}^2
\bigg\}^{\frac12}.
\label{eq:L2 Poincare}
\end{align}
\end{lemma}
\begin{proof}
The proof can be found in Section \ref{sec:proof_L2 Poincare}.
\end{proof}

The $\bL^2$-estimate from Lemma~\ref{lem:L2 Poincare} can be improved by invoking the
elliptic regularity pickup of the Poisson problem with homogeneous Neumann boundary
conditions on the vertex patches. The improvement consists in removing the broken curl contribution from~\eqref{eq:L2 Poincare} and also the global scaling factor $p$ in front of the opening brace. We need, however, a uniform estimate of the underlying
constants over all the vertex patches in the mesh. It is natural to expect that these
constants can be bounded in terms of the mesh shape-regularity. As we did not find a precise
result in the literature supporting this claim, we formulate it as an assumption. We hope
that this will stimulate further research in this direction to prove (or disprove) the
assumption.

\begin{assumption}[Elliptic regularity pickup for Neumann problem] \label{ass:Neumann}
There exist $C(\kappa_{\mesh})$ and $\delta:=\delta(\kappa_{\mesh})>0$ so that, for all $\vtx\in\vertice$ and all $\bg\in \bX\upg(\omvtx):=[X\upg(\omvtx)]^d$, the unique solution $\zeta_{\bg} \in X\upg(\omvtx)$ of the Neumann problem $(\nabla \zeta_{\bg},\nabla \rho)_{\omvtx} = (\bg,\nabla \rho)_{\omvtx}$ for all $\rho \in X\upg(\omvtx)$ (i.e., $\Delta \zeta_{\bg} = \DIV\bg$ in $\omvtx$ and $\bn_{\omvtx } {\cdot}  \nabla \zeta_{\bg}= \bn_{\omvtx} {\cdot} \bg$ on $\partial \omvtx$) satisfies the estimate
\begin{equation} \label{eq:pickup}
\|\zeta_{\bg}\|_{\omvtx} + h_\vtx^{\frac32+\delta} |\nabla \zeta_\bg|_{H^{\frac12+\delta}(\omvtx)}
\le C(\kappa_{\mesh}) h_\vtx^2 \|\nabla \bg\|_{\omvtx}.
\end{equation}
\end{assumption}

\begin{lemma}[Local $\bL^2$-bound using Sobolev embedding on vertex patches]\label{lem:L2 Sobolev}
Under Assumption~\ref{ass:Neumann}, there exists $C(\kappa_{\mesh})$ so that, for all $\bE_h\in \bP\upb_p(\mesh)$ and all $\vtx \in \vertice$,
\begin{align}
\| \bdelta_{\vtx} \|_{\omvtx} \le {}& C(\kappa_{\mesh})\bigg\{
\! \sum_{F\in \Fv } \! h_F \ltwo{ \jump{\bE_h}\upc_F}{F}^2
\bigg\}^{\frac12}.
\label{eq:L2 Sobolev}
\end{align}
\end{lemma}

\begin{proof}
The proof can be found in Section \ref{sec:proof_L2 Sobolev}.
\end{proof}

We are now ready to state global estimates on $\opRec(\bE_h)-\bE_h$.

\begin{theorem}[Bound on $\opRec(\bE_h)-\bE_h$]\label{theorem: nonconforming error total}
There exists $C(\kappa_{\mesh})$ so that, for all $\bE_h\in \bP\upb_p(\mesh)$,
\begin{subequations} \label{eq:main_bounds} \begin{align}
&\|\bROT (\opRec(\bE_h)-\bE_h) \|
\!\leq\! {} C(\kappa_{\mesh}) \bigg\{ \!\su\!
\Big(\!\frac{h_K}{p}\!\Big) \!\ltwo{\jumpK{ \bROT \bE_{h}}\upd }{\dK}^2
\!+\!
\Big( \frac{p^2}{h_K} \!\Big) \!\ltwo{ \jumpK{\bE_h}\upc}{\dK}^2
\! \!\bigg\}^{\frac12}  \!\!,
\label{eq: Global nonconforming error bound weighted broken curl} \\
&\| \opRec(\bE_h)-\bE_h \|
\leq C(\kappa_{\mesh}) p \bigg\{
 \! \su \!
\Big( \frac{h_K}{p}\Big)^3 \! \ltwo{\jumpK{\ROTh \bE_{h}}\upd}{\dK}^2
\!+
h_K \ltwo{ \jumpK{\bE_h}\upc}{\dK}^2 \!\bigg\}^{\frac12}\!\!.\label{eq:global L2 Poincare}
\end{align}
Moreover, under Assumption~\ref{ass:Neumann}, we have
\begin{equation}
\| \opRec(\bE_h)-\bE_h \|
\leq C(\kappa_{\mesh}) \bigg\{ \!\su\! h_K \ltwo{ \jumpK{\bE_h}\upc}{\dK}^2
 \bigg\}^{\frac12}.
\label{eq: Global nonconforming error bound weighted L2}
\end{equation} \end{subequations}
\end{theorem}

\begin{proof}
The bounds~\eqref{eq:main_bounds} readily follow from the local estimates
established in Lemmas~\ref{lem:broken curl}, \ref{lem:L2 Poincare},
and~\ref{lem:L2 Sobolev} once we observe that,
for all $K\in\mesh$, $(\opRec(\bE_h)-\bE_h)|_K = \sum_{\vtx \in \verticeK} \bdelta_{\vtx}|_K$
owing to the partition-of-unity property~\eqref{eq:PU}. Invoking the triangle inequality
and the mesh shape-regularity concludes the proof.
\end{proof}

\begin{remark}[$p=0$] \label{rem:p=0}
The operator $\opRec$ can also be employed on $\bP\upb_0(\mesh) \subset \bP\upb_1(\mesh)$.
\end{remark}

\begin{remark}[$p$-suboptimality]
The $\frac12$-order suboptimality in~\eqref{eq:broken curl} stems from the bound~\eqref{eq:bnd_on_L} on the jump lifting operator (which, itself, follows from the discrete trace inequality~\eqref{eq: discrete_trace_inv}).
The $\frac32$-order suboptimality in \eqref{eq:L2 Poincare} inherits the $\frac12$-order suboptimality from~\eqref{eq:broken curl} to which a one-order suboptimality originating from the broken-curl, divergence-preserving Poincar\'{e} inequality is added. As this latter inequality
does not involve discrete functions, no gain in the polynomial degree can be achieved.
Finally, the $\frac12$-order suboptimality in~\eqref{eq:L2 Sobolev} stems from a Sobolev embedding invoked in a duality argument, which circumvents the need to invoke the
broken curl estimate~\eqref{eq:broken curl}.
\end{remark}

\begin{remark}[Agglomerated polytopal meshes]
Let $\mesh$ be a simplicial background mesh covering $\Dom$ exactly and let $\meshpoly$ be an agglomerated polytopal mesh built from $\mesh$ in such a way that every $\tilde{K}\in \meshpoly$ is the union of a uniformly bounded number of simplicial cells $K\in \mathcal{T}_{\tilde{K}} \subset \mesh$; see Figure \ref{polygonal meshes example} for an illustration. The above assumption on the mesh implies that there exists a constant $C_{\rm sh}\geq 1$ such that
\begin{equation}\label{ass: polytopal meshes}
h_{K} \leq h_{\tilde{K}} \leq C_{\rm sh} h_{K} \quad \forall  K \in \mathcal{T}_{\tilde{K}}.
\end{equation}
Since $\bP\upb_p(\meshpoly)\subset \bP\upb_p(\mesh)$, an $\Hrtz$-reconstruction of any field $\bE_h\in \bP\upb_p(\meshpoly)$ can be devised by considering the $\Hrtz$-reconstruction operator defined on $\bP\upb_p(\mesh)$. Since $\bE_h\in \bP\upb_p(\meshpoly)$ and its broken curl only jump across the interfaces on $\mesh$ that are subsets of interfaces of $\meshpoly$, the error $\opRec(\bE_h)-\bE_h$ can still be bounded as in~\eqref{eq:main_bounds}, both in the broken curl norm and in the $\bL^2$-norm, where the summations on the right-hand side run over the mesh cells in $\meshpoly$ and the local mesh size $h_{\tilde{K}}$ can be used owing to~\eqref{ass: polytopal meshes}. Notice that the condition~\eqref{ass: polytopal meshes} does not accommodate polytopal meshes with tiny faces. This case is left to future work. (We mention \cite{cangiani2023aposteriori} which derives a posteriori estimates for elliptic problems on polytopal meshes with tiny faces, but this work assumes that the domain is simply connected and does not rely on a partition of unity based on a simplicial background mesh.)
\end{remark}

\begin{figure}[tb]
\begin{center}
\includegraphics[width=0.35\linewidth]{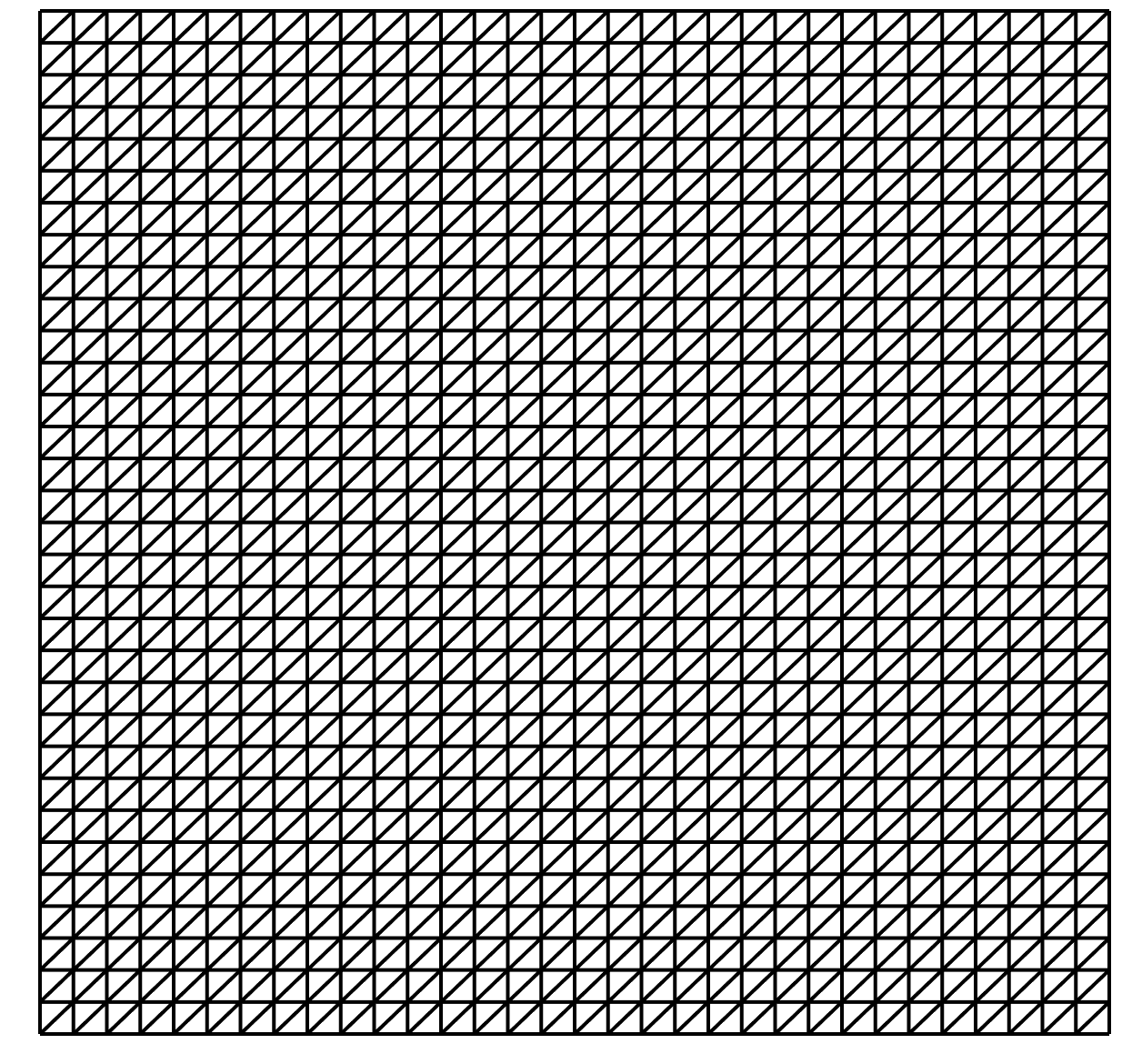}
\qquad
\includegraphics[width=0.35\linewidth]{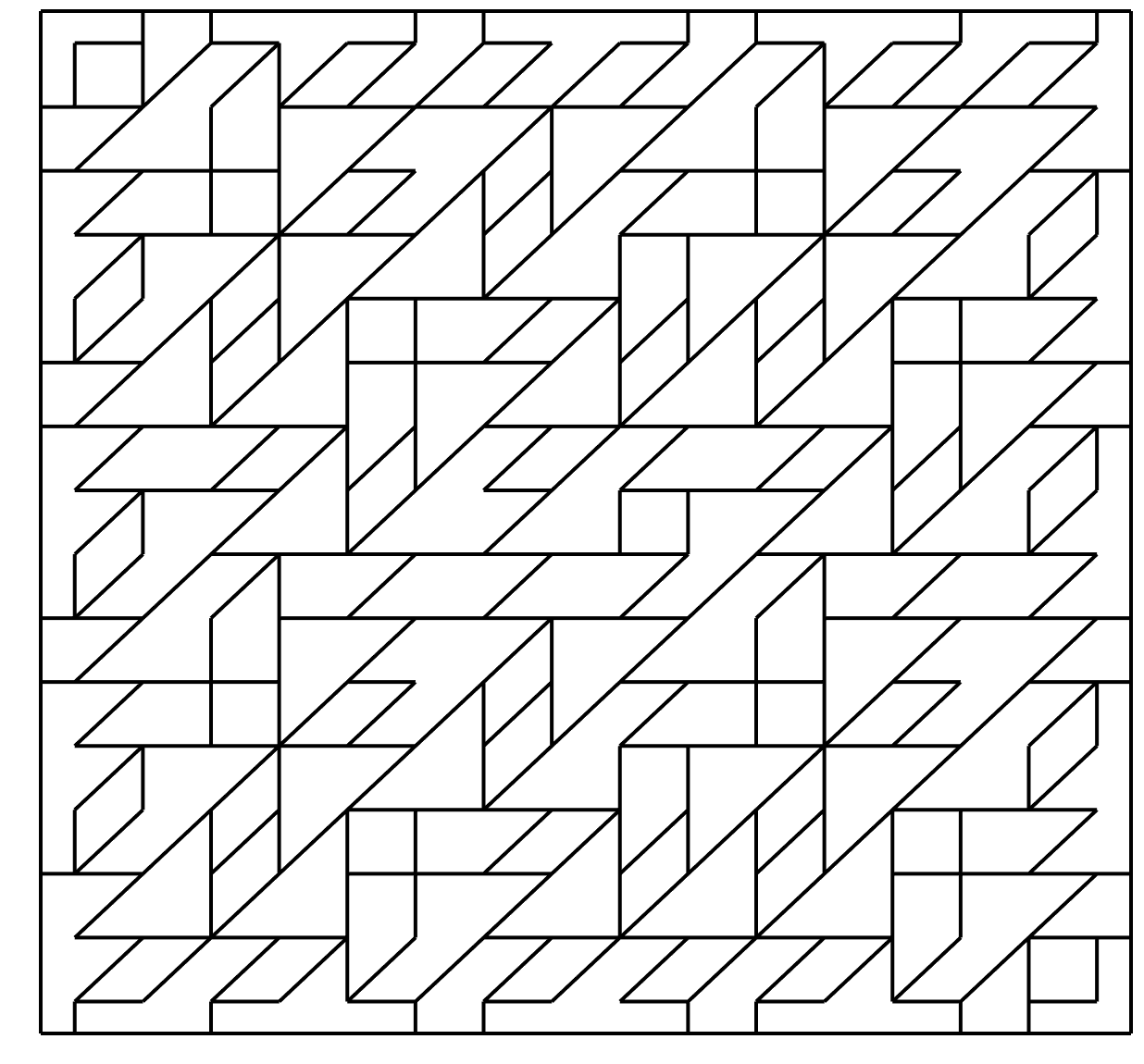}
\end{center}
\caption{Triangular mesh $\mesh$ consisting of 2048 elements (left); agglomerated polygonal mesh $\meshpoly$ consisting of 514 elements (right).}\label{polygonal meshes example}
\end{figure}
\begin{remark}[Mixed boundary conditions] \label{rem:mixed BCs}
The $\Hrtz$-reconstruction from Definition~\ref{def: Patchwise and global potential reconstruction} can be adapted with minor modifications to handle mixed boundary conditions. Let $\Gamma_{\rm D}\subset \partial\Dom$ and $\Gamma_{\rm N} :=  \partial\Dom \backslash \Gamma_{\rm D}$, with $|\Gamma_{\rm N}|\neq0$, be  Lipschitz subsets with connected interior that can be covered exactly by mesh faces.
Let $\HrotzG$ be the subspace of $\Hrot$ prescribing a homogeneous Dirichlet boundary condition only on $\Gamma_{\rm D}$. To devise a global $\HrtzG$-reconstruction, the local problems~\eqref{def: Patch reconstruction} and~\eqref{def: Patch reconstruction curl free} defined on each vertex patch are now posed on the function spaces $\bX_{\Gamma_{\rm D}}\upc(\omvtx):= \bH_{\partial \omvtx \backslash \Gamma_{\rm N}}(\text{\bf curl};\omvtx) \cap \bH_{\partial \omvtx \cap \Gamma_{\rm N}}(\text{\rm  div}=0;\omvtx)$ (where the subscripts indicate where the tangential and normal boundary condition is prescribed, respectively) and $X\upg_{\Gamma_{\rm D}}(\omvtx) := H^1_{\partial \omvtx \backslash \Gamma_{\rm N}}(\omvtx)$, for all $\vtx\in \vertice$. Then one sets again
$\bE_{\vtx}:  = \bU_{\vtx}  + \nabla \theta_{\vtx}$,
and extending $\bE_{\vtx}$ by zero to $\Dom$, one sets
$\opRec(\bE_h): = \sum_{\vtx\in \vertice} \bE_{\vtx}$, so that $\opRec(\bE_h)\in \HrotzG$.
The local broken curl estimate now relies on the local Helmholtz decomposition with mixed boundary conditions discussed in Remark \ref{Mixed BCs for Helmholtz decomposition} below. The estimate reads as follows:
$$
\|\bROT \bdelta_{\vtx}\|_{\omvtx} \le {}
C(\kappa_{\mesh}) \bigg\{
 \! \sum_{F\in \Fv \backslash \FN  } \!
 \Big(\frac{h_F}{p}\Big) \ltwo{\jump{ \bROT \bE_h}\upd_F }{F}^2
+ \Big(\frac{p^2}{h_F} \Big) \ltwo{ \jump{\bE_h}\upc_F}{F}^2
\bigg\}^{\frac12},
$$
where $\FN$ denotes the subset of mesh boundary faces lying on $\Gamma_{\rm N}$.
The local $\bL^2$-error estimate relies on the broken Poincar\'{e} with mixed boundary conditions discussed in Remark \ref{Mixed BCs for Poincare inequality} below (which requires one more conjecture, unless pure Neumann conditions are enforced).  The estimate reads as follows:
$$
\| \bdelta_{\vtx} \|_{\omvtx} \le {}  C(\kappa_{\mesh}) p \bigg\{
 \! \sum_{F\in \Fv \backslash \FN  } \!
\Big( \frac{h_F}{p}\Big)^3 \ltwo{\jump{\ROTh \bE_h}\upd_F }{F}^2
+ h_F \ltwo{ \jump{\bE_h}\upc_F}{F}^2
\bigg\}^{\frac12}.
$$
Finally, an extension of the bound from Lemma~\ref{lem:L2 Sobolev} to mixed boundary conditions is not expected to hold because the elliptic regularity pickup stated in Assumption~\ref{ass:Neumann} is not expected to hold.
\end{remark}

\begin{remark}[Globally defined $\Hrtz$-reconstruction and convex domains]
One can also consider a globally defined $\Hrtz$-reconstruction operator, say
$\optRec:\bP\upb_p(\mesh)\rightarrow \Hrotz$, without posing auxiliary problems
on vertex patches exploiting the partition of unity. Assume first that $\Dom$
has trivial topology. Then, one considers
$\bU_{\Dom}\in \bXz:=\Hrotz\cap \Hdivdivz$ solving the following well-posed problem:
\begin{equation*}
(\ROTZ \bU_{\Dom}, \ROTZ  \bW_{\Dom}) = (\bROT \bE_{h} + \bLhpz(\bE_h), \ROTZ \bW_{\Dom})
\quad \forall  \bW_{\Dom} \in  \bXz,
\end{equation*}
and $\theta_{\Dom}\in H^1_0(\Dom)$ solving the following well-posed problem:
\begin{equation*}
( \nabla \theta_{\Dom}, \nabla v_{\Dom}  ) = (\bE_h,\nabla v_{\Dom})
\quad \forall  v_{\Dom} \in  H^1_0(\Dom),
\end{equation*}
and one sets $\optRec(\bE_h):= \bU_{\Dom}  + \nabla \theta_{\Dom}$.
The error $\optRec(\bE_h)-\bE_h$ can again be bounded by the right-hand sides in~\eqref{eq:main_bounds}, both in the broken curl norm and in the $\bL^2$-norm. The proofs, which follow similar arguments to those presented in Section~\ref{sec:proofs} are omitted for brevity.
One advantage of the globally defined $\Hrtz$-reconstruction operator is achieved when the domain $\Dom$ is convex. Indeed, in this case, full elliptic regularity pickup can be invoked when deriving the $\bL^2$-error bound, leading to the $hp$-optimal estimate
\[
\| \optRec(\bE_h)-\bE_h \|
\leq C(\kappa_{\mesh}) \bigg\{ \!\su\! \Big( \frac{h_K}{p}\Big) \ltwo{ \jumpK{\bE_h}\upc}{\dK}^2 \bigg\}^{\frac12}.
\]
However, the broken curl norm bound cannot be improved and remains $p$-suboptimal by $\frac12$-order. Finally, the above global construction can also be deployed if the domain $\Dom$ has a more general topology upon posing the problem defining $\bU_\Dom$ in $\Hrotz\cap \Hrotzrotz^\perp$, where $\Hrotzrotz$ is the subspace of $\Hrotz$ composed of curl-free fields and orthogonalities are meant in $\bL^2$. However, the constants in the bounds on $\optRec(\bE_h)-\bE_h$ now depend on the topology of $\Dom$.
\end{remark}

\section{Preparatory results} \label{sec:preparatory}

In this section, we collect three preparatory results. The first result is a (classical)
local Helmholtz decomposition to be invoked in the proof of Lemma~\ref{lem:broken curl}.
The second one is a local broken-curl, divergence-preserving Poincar\'e inequality to be invoked in the proof of
Lemma~\ref{lem:L2 Poincare}. Finally, the third one is a local Sobolev embedding to be
invoked in the proof of Lemma~\ref{lem:L2 Sobolev}.

\begin{lemma}[Local Helmholtz decomposition] \label{lemma: Helmholtz decomposition}
For all $\vtx\in\vertice$ and all $\bv \in \bL^2(\omvtx)$, there exist $\bpsi \in \bXzv$
(see~\eqref{eq:def_bXzv}) and $\xi\in X\upg(\omvtx)$ (see~\eqref{def: H1_patch}),
so that
\begin{subequations} \label{eq:Helmholtz} \begin{align} \label{eq:helm curl}
\bv &=  \ROTZ  \bpsi +\nabla \xi \quad \text{in }\, \omvtx, \\
\ltwo{\bv}{\omvtx}^2&= \ltwo{ \ROTZ  \bpsi }{\omvtx}^2 +\ltwo{ \nabla  \xi }{\omvtx}^2.
\label{eq: new HD stability}
\end{align}\end{subequations}
\end{lemma}
\begin{proof}
For completeness, we sketch a short proof.
We solve the Neumann problem
$
(\nabla \xi , \nabla w)_{\omvtx} =(\bv, \nabla w)_{\omvtx}
$
in $X\upg(\omvtx)$, and observe that
$
\nabla \xi - \bv \in \Hdivzvz
$. Using \cite[Theorem 3]{HP19_822}, we infer that there exists $\bzeta \in \Hrotzv$ such that
$
\ROTZ \bzeta = \bv - \nabla \xi .
$
Then, we solve the Dirichlet problem
$
( \nabla \gamma , \nabla w)_{\omvtx} =( \bzeta, \nabla w)_{\omvtx}
$
in $X\upg_0(\omvtx)$ and set $\bpsi = \bzeta - \nabla \gamma$. The pair $(\bpsi, \xi)$ satisfies all the required properties.
\end{proof}

\begin{remark}[Mixed boundary conditions]\label{Mixed BCs for Helmholtz decomposition}
For mixed boundary conditions, Lemma~\ref{lemma: Helmholtz decomposition} is adapted as follows:
For all $\vtx\in\vertice$ and all $\bv \in \bL^2(\omvtx)$, there exist $\bpsi \in \bX_{\Gamma_{\rm D}}\upc (\omvtx)$ and $\xi\in H^1_{\partial \omvtx \cap \Gamma_{\rm N}}(\omvtx)$,
so that \eqref{eq:Helmholtz} still holds true.
To prove this result, one solves Poisson problems with mixed boundary conditions.
\end{remark}

\begin{lemma}[Local broken-curl, divergence-preserving Poincar\'{e} inequality] \label{lemma: broken_Poincare inequality}
 There exists a constant $C(\kappa_{\mesh})$ so that, for all $\vtx \in\vertice$, all $\bv\in  \bH^1(\meshv):=  \{\bw \in \bL^2(\omvtx)  \st \GRAD_h \bw\in \bL^2(\omvtx)\}$, and all $\bv\upc_{\vtx}\in \Hrotzv$ with $\bv\upc_{\vtx}-\bv\in \Hdivvz$, the following holds:
\begin{equation}\label{eq:broken Poincare}
\| \bv\upc_{\vtx} - \bv\|_{\omvtx} \le {}
C(\kappa_{\mesh}) h_{\vtx} \bigg\{ \|\bROT (\bv\upc_{\vtx}-\bv) \|_{\omvtx}^2
+
\! \sum_{F\in \bFa } \!
h_F^{-1} \ltwo{ \jump{\bv}\upc_F}{F}^2 \bigg\}^{\frac12},
\end{equation}
recalling that the jump operator at $F$ denotes the actual trace whenever $F\in\Fb$.
\end{lemma}

\begin{proof}
Let $\vtx\in\vertice$, let $\bv\in \bH^1(\meshv)$, and let $\bv\upc_{\vtx}\in \Hrotzv$ be such that $\bv\upc_{\vtx}-\bv\in \Hdivvz$. Recall that
$X\upg(\omvtx):=H^1(\omvtx)/\Real$.

(1) Since $\omvtx$ is a simply connected, Lipschitz domain, we infer from \cite[Lemma 3.5]{Amrouche98} that there exists $\bphi \in  \bX\upg(\omvtx):=[X\upg(\omvtx)]^d$ such that
$\ROT \bphi = \bv\upc_{\vtx} - \bv$ in $\omvtx$ (notice that $\DIV \bphi = 0$ in $\omvtx$ as well, but this additional property is not needed here).
However, the stability constant on $\ltwo{\nabla \bphi}{\omvtx}$ is not known explicitly in terms of the geometry of $\omvtx$.
For interior vertices, the vertex patch is a star-shaped domain with respect to a ball whose radius is comparable to the size of the patch, with constant depending only on the mesh regularity. This result is established in \cite[Proposition 8.2]{chaumontfrelet:hal-05204325} in any dimension $d\ge2$ by reduction to the case $d=2$ where the result is known from \cite{LeePre:79}. Invoking \cite[Corollary 28]{GuzSal21} then gives $\bphi \in  \bX\upg(\omvtx)$ such that
$\ROT \bphi = \bv\upc_{\vtx} - \bv$ in $\omvtx$ and a stability constant $C(\kappa_{\mesh})$ only depending on
the mesh shape-regularity so that the following bound holds:
\begin{equation}
\ltwo{\nabla \bphi}{\omvtx}\leq C(\kappa_{\mesh}) \ltwo{\ROT \bphi}{\omvtx}
= C(\kappa_{\mesh})\ltwo{\bv\upc_{\vtx}-\bv}{\omvtx}. \label{eq:curl norm}
\end{equation}
For boundary vertices, the star-shapedness of the patch can be affected by the domain boundary so that the radius of the ball can deteriorate, e.g., near re-entrant corners. The way forward is to view each boundary vertex patch as a chain of star-shaped domains, and invoke \cite[Theorem 35]{GuzSal21} to obtain again the bound~\eqref{eq:curl norm}. Following \cite[Appendix B]{ErnVohralik2020}, we use the notion of shellability and enumerate the simplices in the patch as $\{K_i\}_{i=1}^{n}$ so that two successive simplices in the enumeration share a face, say $F_i=\partial K_i\cap \partial K_{i+1}$. We then define the subdomains $\omvtx^i := K_i \cup K_{i+1}$ for all $i\in\{1{:}n{-}1\}$. As a consequence of mesh shape-regularity, each $\omvtx^i$ is star-shaped with respect to a $d$-dimensional ball centered at the barycenter of the face $F_i$, with radius comparable to $h_{F_i}$. Moreover, we have the overlap relation $K_{i+1} = \omvtx^i \cap \omvtx^{i+1}$. All the conditions stated at the beginning of \cite[Section 5]{GuzSal21} are thus fulfilled to consider the boundary vertex patch as a chain of star-shaped domains.

(2) Integrating by parts the curl operator in each cell $K\in \meshv$,
we infer that
\begin{equation*}
\begin{split}
\ltwo{\bv\upc_{\vtx} - \bv}{\omvtx}^2
&= (\bv\upc_{\vtx} - \bv,  \ROT  \bphi)_{\omvtx}\\
& = (\bv\upc_{\vtx} ,\ROT\bphi)_{\omvtx}  - (\bv,\ROT \bphi)_{\omvtx} \\
& = - ( \bROT (\bv\upc_{\vtx}-\bv), \bphi  )_{\omvtx} - \suv (\bv,\n_K \times\bphi)_{\dK}  =: T_1+T_2,
\end{split}
\end{equation*}
and it remains to bound $T_1$ and $T_2$.

(2a) Bound on $T_1$. Using the Cauchy--Schwarz inequality, recalling that $\bphi$ has zero mean-value over $\omvtx$, and invoking the Poincar\'{e} inequality on the vertex patch (see \cite[Theorem 3.2]{VeeserVerfurth12} and also \cite[Lem.~5.7]{ErnGu:17_quasi}),  we infer that
$$
|T_{1}|
\lesssim h_{\vtx} \ltwo{\bROT (\bv\upc_{\vtx} - \bv)}{\omvtx} \ltwo{ \nabla \bphi}{\omvtx}
\lesssim h_{\vtx} \ltwo{\bROT (\bv\upc_{\vtx} - \bv)}{\omvtx} \ltwo{\bv\upc_{\vtx}-\bv}{\omvtx},
$$
where the second bound follows from~\eqref{eq:curl norm}.

(2b) Bound on $T_2$. Re-organizing the summation over the mesh faces in $\bFa$ as those in $\bFa\cap \Fint$ where $\bphi$ has continuous trace and those in $\bFa\cap \Fb$ where the jump operator denotes the actual trace, we obtain
\[
T_2 = \! \sum_{F\in \bFa} \! (\jump{\bv}\upc_F,\n_F \times\bphi)_{F}.
\]
Invoking the same arguments as above together with the mesh shape-regularity and
a multiplicative trace inequality
(see, e.g., \cite[Lem.~12.15]{EG_volI}), we infer that
\begin{equation*}
\begin{split}
|T_{2}|
&\leq {}  \bigg\{
 \! \sum_{F\in \bFa } \!
h_F^{-1} \ltwo{ \jump{\bv}\upc_F}{F}^2 \bigg\}^{\frac12}
\bigg\{
 \! \sum_{F\in \bFa}  \!
 h_F \ltwo{\bphi}{F}^2
\bigg\}^{\frac12} \\
&\lesssim h_{\vtx}\bigg\{
\! \sum_{F\in  \bFa} \!
h_F^{-1} \ltwo{ \jump{\bv}\upc_F}{F}^2 \bigg\}^{\frac12}
\ltwo{\bv\upc_{\vtx} - \bv}{\omvtx}.
\end{split}
\end{equation*}
Putting the bounds on $T_1$ and $T_2$ together proves \eqref{eq:broken Poincare}.
\end{proof}

\begin{remark}[Extension]
A simple inspection of the proof shows that Lemma~\ref{lemma: broken_Poincare inequality} holds true for any $\bv\in \bL^2(\omvtx)$ with $\ROTh \bv\in \bL^2(\omvtx)$ and
$(\bn_K\times \bv|_K)|_F\in \bL^2(F)$ for all $K\in\meshv$ and all $F\in\calFK$.
\end{remark}

\begin{remark}[Mixed boundary conditions]\label{Mixed BCs for Poincare inequality}
For mixed boundary conditions, the statement of Lemma~\ref{lemma: broken_Poincare inequality} is adapted as follows:
There exists a constant $C(\kappa_{\mesh})$ so that, for all $\vtx \in\vertice$, all $\bv\in  \bH^1(\meshv)$, and all $\bv\upc_{\vtx}\in \bH_{\partial \omvtx \backslash \Gamma_{\rm N}}(\text{\bf curl};\omvtx)$ with $\bv\upc_{\vtx}-\bv\in \bH_{\partial \omvtx \cap \Gamma_{\rm N}}(\text{\rm  div}=0;\omvtx)$, the following holds:
\begin{equation*}
\| \bv\upc_{\vtx} - \bv\|_{\omvtx} \le
C(\kappa_{\mesh}) h_{\vtx} \bigg\{ \|\bROT (\bv\upc_{\vtx}-\bv) \|_{\omvtx}^2
+
\! \sum_{F\in \bFa \backslash \FN} \!
h_F^{-1} \ltwo{ \jump{\bv}\upc_F}{F}^2 \bigg\}^{\frac12}.
\end{equation*}
The proof of this inequality is, however, not straightforward. First, \cite[Lemma 3.5]{Amrouche98} can be replaced by \cite[Theorem 3]{HP19_822} on regular decompositions with mixed boundary conditions, but the stability constant is not given explicitly in terms of the geometry of $\omvtx$. To have a stability constant only depending on the mesh shape-regularity, one needs to conjecture that the result of \cite{GuzSal21}  can be extended to mixed boundary conditions. We notice that the extension holds for pure Neumann  conditions \cite[Theorem 32]{GuzSal21}.
\end{remark}

\begin{remark}[Local broken-divergence, curl-preserving Poincar\'e inequality]
\!One can adapt the proof of Lemma~\ref{lemma: broken_Poincare inequality} to establish the following local broken-divergence, curl-preserving Poincar\'e inequality:
There exists a constant $C(\kappa_{\mesh})$ so that, for all $\vtx \in\vertice$, all $\bv\in  \bH^1(\meshv)$, and all $\bv\upc_{\vtx}\in \Hdivzv$ with $\bv\upc_{\vtx}-\bv\in \Hrotvz$, the following holds:
\begin{equation*}
\| \bv\upc_{\vtx} - \bv\|_{\omvtx} \le {}
C(\kappa_{\mesh}) h_{\vtx}\bigg\{ \|\DIVh (\bv\upc_{\vtx}-\bv) \|_{\omvtx}^2
+
\! \sum_{F\in \bFa } \!
h_F^{-1} \ltwo{ \jump{\bv}\upd_F}{F}^2
\bigg\}^{\frac12}.
\end{equation*}
\end{remark}

\begin{lemma}[Local Sobolev embedding]\label{lemma: regularity}
Under Assumption~\ref{ass:Neumann} with $\delta:=\delta(\kappa_{\mesh})
\in(0,\frac12]$,
there exists a constant $C(\kappa_{\mesh})$ (depending on $\delta$) such that, for all $\vtx\in\vertice$ and all $\bv \in \bXdivzv$ (see~\eqref{eq:def_bXdivzv}),
\begin{equation}\label{eq: regularity on X}
|\bv|_{\bH^{\frac12 +\delta}(K)}
\leq C(\kappa_{\mesh}) h_{\vtx}^{\frac12-\delta} \ltwo{\ROT \bv}{\omvtx} \quad\forall K\in\meshv.
\end{equation}

\end{lemma}

\begin{proof}
(1) Let $\vtx\in\vertice$ and $\bv \in \bXdivzv$. Invoking the same arguments as in Step~(1) of the proof of Lemma~\ref{lemma: broken_Poincare inequality}, we infer that there exists $\bphi \in \bX\upg(\omvtx)$ such that $\ROT \bphi = \ROT \bv$
in $\omvtx$, and
\begin{equation}\label{relation 1}
\ltwo{\nabla \bphi}{\omvtx}\leq C(\kappa_{\mesh}) \ltwo{\ROT \bphi}{\omvtx} =
C(\kappa_{\mesh}) \ltwo{\ROT \bv}{\omvtx}.
\end{equation}
Since $\ROT(\bv-\bphi)=\bzero$, there exists a function $p\in X\upg(\omvtx)$ such that $\nabla p = \bphi - \bv$. Altogether, this leads to the regular decomposition
\[
\bv = -\nabla p + \bphi, \quad p\in X\upg(\omvtx), \quad \bphi \in \bX\upg(\omvtx).
\]

(2) Since $\bv \in \Hdivzdivz$ by assumption, we infer
that $(\nabla p,\!\nabla \rho)_{\omvtx} \!\!=\! (\bphi,\!\nabla\rho)_{\omvtx}$ for all $\rho
\in  X\upg(\omvtx)$. Therefore, we can apply Assumption~\ref{ass:Neumann}
to $\bg:=\bphi$, and observing that $\zeta_{\bg}=p$, we infer that
\[
h_\vtx^{\frac32+\delta} |\nabla p|_{\bH^{\frac12+\delta}(\omvtx)} \lesssim h_\vtx^2\|\nabla\bphi\|_{\omvtx}
\lesssim h_\vtx^2\|\ROT\bv\|_{\omvtx},
\]
where the second bound follows from~\eqref{relation 1}. Since $|\nabla p|_{\bH^{\frac12+\delta}(K)}\le |\nabla p|_{\bH^{\frac12+\delta}(\omvtx)}$ for all $K\in\meshv$, we infer that
\begin{equation} \label{eq:bnd_gradp_Sobolev}
|\nabla p|_{\bH^{\frac12+\delta}(K)} \lesssim  h_\vtx^{\frac12-\delta}\|\ROT\bv\|_{\omvtx}.
\end{equation}

(3) For all $K\in\meshv$, invoking the geometric mapping to the reference simplex $\wK$
and the embedding $|\wbphi|_{\bH^{\frac12 +\delta}(\wK)} \le \wC \|\GRAD \wbphi\|_{\wK}$, we infer that
\[
|\bphi|_{\bH^{\frac12 +\delta}(K)} \lesssim h_K^{\frac12-\delta}\|\nabla\bphi\|_K
\le h_K^{\frac12-\delta}\|\nabla\bphi\|_{\omvtx}
\lesssim  h_\vtx^{\frac12-\delta}\|\ROT\bv\|_{\omvtx},
\]
where the last bound follows again from~\eqref{relation 1} and $h_K\le h_\vtx$.
Combining this bound with~\eqref{eq:bnd_gradp_Sobolev} and invoking the triangle inequality concludes the proof.
\end{proof}

\section{Proof of the main results from Section~\ref{sec:main results}} \label{sec:proofs}

This section contains the proofs of Lemmas~\ref{lem:broken curl}, \ref{lem:L2 Poincare},
and~\ref{lem:L2 Sobolev}.

\subsection{Proof of Lemma~\ref{lem:broken curl}} \label{sec:proof_broken curl}

Let $\vtx\in\vertice$. Our goal is bound the error $\bdelta_{\vtx}:= \bE_{\vtx} - \psi_{\vtx} \bE_h$ defined in~\eqref{eq:encvtx} in the broken curl norm. Specifically, our goal is to prove that
\[
\|\bROT \bdelta_{\vtx} \|_{\omvtx} \le {}
C \bigg\{
 \! \sum_{F\in \Fv   } \!
\Big(\frac{h_F}{p}\Big) \ltwo{\jump{ \ROT   \bE_{h}}\upd_F }{F}^2
+\! \sum_{F\in \Fv } \!
\Big(\frac{p^2}{h_F}\Big) \ltwo{ \jump{\bE_h}\upc_F}{F}^2
\bigg\}^{\frac12}.
\]
To this purpose, we apply the Helmholtz decomposition \eqref{eq:helm curl} to $\bv:=\bROT \bdelta_{\vtx}$ in the vertex patch $\omvtx$. This
gives $\bpsi \in  \bXzv$ and $\xi\in X^{\rm g} (\omvtx)$ such that
$\bROT \bdelta_{\vtx}=  \ROTZ  \bpsi +\nabla \xi$ in $\omvtx$. Taking the $\bL^2(\omvtx)$-inner product with $\bROT \bdelta_{\vtx}$, we infer that
\begin{equation}\label{eq: nonconforming bound}
\ltwo{\bROT \bdelta_{\vtx}}{\omvtx}^2
=  ( \bROT \bdelta_{\vtx},  \ROTZ  \bpsi )_{\omvtx}
+( \bROT \bdelta_{\vtx},  \nabla \xi)_{\omvtx} =: \mbox{I} +  \mbox{II},
\end{equation}
and it remains to bound $\mbox{I}$ and $\mbox{II}$.

(1) Bound on $\mbox{I}$. Since $\bROT \bdelta_{\vtx} = \ROTZ \bE_{\vtx} - \bROT  (\psi_{\vtx} \bE_{h}) =  \ROTZ \bU_{\vtx} - \bROT  (\psi_{\vtx} \bE_{h})$ owing to the definition \eqref{def: vertex curl reconstruction} of $\bE_{\vtx}$, we have
\[
\mbox{I}= ( \ROTZ \bU_{\vtx} - \bROT  (\psi_{\vtx} \bE_{h}),  \ROTZ  \bpsi )_{\omvtx}.
\]
Recalling the definition \eqref{def: Patch reconstruction} of $\bU_{\vtx}$ and since
$\bpsi \in   \bXzv$, this gives
\[
\mbox{I}= ( \bLhpz(\psi_{\vtx} \bE_h) ,  \ROTZ  \bpsi )_{\omvtx}.
\]
Invoking the Cauchy--Schwarz inequality and the fact that $\psi_{\vtx}$ is continuous and bounded by one, we infer that
\begin{equation}\label{eq: curl bound 1}
\begin{split}
\mbox{I}\leq \ltwo{\bLhpz(\psi_{\vtx} \bE_h) }{\omvtx} \ltwo{\ROTZ \bpsi}{\omvtx}
&\lesssim \bigg\{
\! \sum_{F\in \Fv } \!
\Big(\frac{p^2}{h_F}\Big) \ltwo{ \jump{\psi_\vtx\bE_h}\upc_F}{F}^2
\bigg\}^{\frac12} \ltwo{\ROTZ\bpsi}{\omvtx}\\
&\lesssim \bigg\{
\! \sum_{F\in \Fv } \!
\Big(\frac{p^2}{h_F}\Big) \ltwo{ \jump{\bE_h}\upc_F}{F}^2
\bigg\}^{\frac12} \ltwo{\bROT \bdelta_{\vtx}}{\omvtx},
\end{split}
\end{equation}
where we used the stability bound~\eqref{eq: new HD stability} in the last estimate.

(2) Bound on $\mbox{II}$. Since $(\ROTZ \bE_{\vtx},\nabla \xi)_{\omvtx}=0$, we have
\[
\mbox{II} = -(\bROT(\psi_{\vtx}\bE_h) ,\nabla \xi)_{\omvtx}.
\]
Let $\Igv$ be the modified $hp$-Karkulik--Melenk interpolation operator on the vertex patch $\omvtx$ from Lemma~\ref{lemma:local_hp-KM}. We have
\[
\mbox{II} =
-( \bROT (\psi_{\vtx} \bE_h),  \nabla (\xi - \Igv(\xi)))_{\omvtx}
- ( \bROT (\psi_{\vtx} \bE_h),  \nabla \Igv(\xi))_{\omvtx} =: \mbox{II}_1 + \mbox{II}_2,
\]
and we bound the two terms on the right-hand side.

(2a) Bound on $\mbox{II}_1$.
We integrate by parts the gradient operator in all the cells composing $\meshv$.
This gives
\[
\mbox{II}_1
= -\suv (\xi - \Igv (\xi),\bn_K\SCAL(\bROT(\psi_{\vtx} \bE_h)))_{\partial K}.
\]
All the mesh
faces in $\bigcup_{K\in\meshv} \calFK$ are either in $\Fv$ or in $\Fma$. In the first
case, we observe that $\xi - \Igv  (\xi)$ is single-valued across every interface $F\in\Fv\cap \Fint$ and that we have set conventionally $\jump{\bullet}\upd_F:=\bullet|_F\SCAL\bn_\Dom$ on every boundary face $F\in\Fv\cap \Fb$.
In the second case where $F\in \Fma$, we have
$\bn_F{\cdot} \bROT(\psi_{\vtx} \bE_{h}) = \bn_F{\cdot} (\psi_\vtx {\bROT\bE_h})+ {\bE_h}{\cdot}(\bn_F {\times} \GRAD\psi_\vtx  ) = 0$ since $\psi_{\vtx} = 0$ and $(\n_F {\times} \nabla \psi_{\vtx} )_F = \bzero$ on $F$.
Altogether, this gives
\begin{align*}
\mbox{II}_1
= {}& - \sum_{F\in \Fv} (\xi - \Igv (\xi),  \jump{\bROT  (\psi_{\vtx} \bE_{h})}\upd_F )_{F} .
\end{align*}
We have $\bROT  (\psi_{\vtx} \bE_{h})
= \psi_{\vtx}\bROT \bE_{h} + (\nabla\psi_{\vtx}) \times \bE_{h}$
so that
$$\jump{\bROT\psi_{\vtx} \bE_{h}}\upd_F = \psi_\vtx\jump{\bROT\bE_h}\upd_F+\GRAD\psi_\vtx\SCAL\jump{\bE_h}\upc_F,
$$
since
$\psi_{\vtx}$ is continuous and $(\n_F {\times} \nabla \psi_{\vtx} )_F$ is single-valued across every interface $F\in\Fv \cap \Fint$. This gives
\begin{align*}
\mbox{II}_1
= {}& -\!\sum_{F\in \Fv} \!(\xi - \Igv (\xi), \psi_{\vtx}   \jump{\bROT \bE_{h}}\upd_F )_{F}
 - \!\sum_{F\in \Fv}\! (\xi - \Igv (\xi),\nabla \psi_{\vtx} \SCAL \jump{\bE_{h}}\upc_F )_{F}.
\end{align*}
Let $\mbox{II}_{11}$ and $\mbox{II}_{12}$ denote the two summations on the right-hand side.
Using the Cauchy--Schwarz inequality and $\|\psi_{\vtx}\|_{L^\infty(F)}=1$, we infer that
\begin{align*}
|\mbox{II}_{11}|
\leq {}&  \bigg\{ \! \sum_{F\in \Fv } \!
\Big(\frac{h_F}{p}\Big) \ltwo{\jump{\bROT \bE_{h}}\upd_F }{F}^2 \bigg\}^{\frac12}
 \bigg\{ \!  \sum_{F\in \Fv  } \!
\Big(\frac{p}{h_F}\Big)\ltwo{\xi - \Igv  (\xi)}{F}^2 \bigg\}^{\frac12}.
\end{align*}
We now invoke the approximation result \eqref{eq:inter_Ig on patch}
on $\Igv$. For every $F\in \Fv$, we can pick a mesh cell $K\in\meshv$ of which $F$ is a face and obtain
\[
\Big(\frac{p}{h_F}\Big)^{\frac12}\ltwo{\xi - \Igv (\xi)}{F}
\lesssim \|\nabla \xi\|_{K} \le \|\nabla \xi\|_{\omvtx} \le \ltwo{\bROT \bdelta_{\vtx}}{\omvtx},
\]
where we used~\eqref{eq: new HD stability} in the last bound.
This gives
\[
|\mbox{II}_{11}|
\lesssim \bigg\{ \! \sum_{F\in \Fv  } \!
\Big(\frac{h_F}{p}\Big) \ltwo{\jump{\bROT \bE_{h}}\upd_F }{F}^2
 \bigg\}^{\frac12}
\ltwo{\bROT \bdelta_{\vtx}}{\omvtx}.
\]
Moreover, using the Cauchy--Schwarz inequality and $\| \nabla \psi_{\vtx}\|_{L^\infty(F)}\lesssim h^{-1}_F$ gives
\begin{align*}
|\mbox{II}_{12}| \lesssim {}&
\bigg\{ \!\sum_{F\in \Fv   }\! \Big(\frac{h_F}{p}\Big) h_F^{-2}\ltwo{ \jump{ \bE_{h}}\upc_F }{F}^2
\bigg\}^{\frac12}
 \bigg\{ \!  \sum_{F\in \Fv  } \!
\Big(\frac{p}{h_F}\Big)\ltwo{\xi - \Igv  (\xi)}{F}^2 \bigg\}^{\frac12}.
\end{align*}
Invoking the same arguments as above to bound the factor involving $\xi$, we obtain
\begin{align*}
|\mbox{II}_{12}|\lesssim {}& \bigg\{ \!\sum_{F\in \Fv   }\! p^{-3}\Big(\frac{p^2}{h_F}\Big) \ltwo{ \jump{ \bE_{h}}\upc_F }{F}^2
\bigg\}^{\frac12} \ltwo{\bROT \bdelta_{\vtx}}{\omvtx}.
\end{align*}
Putting everything together, we conclude that
\begin{align*}
|\mbox{II}_1| \lesssim {}&  \bigg\{ \! \sum_{F\in \Fv } \!
\Big(\frac{h_F}{p}\Big) \ltwo{\jump{\bROT \bE_{h}}\upd_F }{F}^2
+ p^{-3} \Big(\frac{p^2}{h_F}\Big)
\ltwo{ \jump{ \bE_{h}}\upc_F }{F}^2  \bigg\}^{\frac12}
  \ltwo{\bROT \bdelta_{\vtx}}{\omvtx} .
\end{align*}

(2b) Bound on $\mbox{II}_2$.
This time we integrate by parts the curl operator in all the cells composing $\meshv$.
Using similar arguments as above, we obtain
\[
\mbox{II}_2
= \suv (\psi_\vtx \bE_{h},   \n_K \times \nabla \Igv(\xi)  )_{\dK}
= \sum_{F\in \Fv} ( \psi_{\vtx} \jump{  \bE_{h}}\upc_F,   \nabla \Igv(\xi) )_{F} .
\]
Invoking the Cauchy--Schwarz inequality,
the discrete trace inequality \eqref{eq: discrete_trace_inv}, and $\|\psi_{\vtx}\|_{L^\infty(F)}=1$,
we infer that
\begin{align*}
|\mbox{II}_2| \leq {}& \bigg\{\!\sum_{F\in \Fv} \! \Big(\frac{p^2}{h_K}\Big) \ltwo{ \jump{  \bE_{h}}\upc_F }{F}^2  \bigg\}^{\frac12}
\bigg\{ \!\sum_{F\in \Fv}  \Big( \frac{h_K}{p^2} \Big) \ltwo{ \nabla \Igv(\xi)}{F }^2  \bigg\}^{\frac12}  \\
\lesssim {}&\bigg\{\!\sum_{F\in \Fv} \! \Big( \frac{p^2}{h_K}\Big) \ltwo{ \jump{  \bE_{h}}\upc_F }{F}^2  \bigg\}^{\frac12}
\bigg\{ \!\suv\!  \ltwo{ \nabla \Igv(\xi)}{K }^2  \bigg\}^{\frac12} \\
\lesssim {}& \bigg\{\!\sum_{F\in \Fv} \! \Big(\frac{p^2}{h_K}\Big)\ltwo{ \jump{  \bE_{h}}\upc_F }{F}^2  \bigg\}^{\frac12}
 \ltwo{ \nabla \xi}{\omvtx},
\end{align*}
where the last bound follows from the $H^1$-stability of $\Igv$
(see~\eqref{eq:inter_Ig on patch}).
Proceeding as above to bound $\|\nabla \xi\|_{\omvtx}$, we conclude that
\[
|\mbox{II}_2| \lesssim \bigg\{\!\sum_{F\in \Fv} \! \Big(\frac{p^2}{h_K}\Big)\ltwo{ \jump{  \bE_{h}}\upc_F }{F}^2  \bigg\}^{\frac12} \ltwo{\bROT \bdelta_{\vtx}}{\omvtx} .
\]
(3) Putting the bounds on $\mbox{II}_1$ and $\mbox{II}_2$ together and since $p\ge1$, we conclude the proof. (Notice that we just drop the factor $p^{-3}$ in the bound on $\mbox{II}_1$.)

\subsection{Proof of Lemma~\ref{lem:L2 Poincare}}
\label{sec:proof_L2 Poincare}

We want to prove that
\[
\| \bdelta_{\vtx} \|_{\omvtx} \lesssim p \bigg\{
 \! \sum_{F\in \Fv   } \!
\Big( \frac{h_F}{p}\Big)^3 \ltwo{\jump{\ROTh \bE_h}\upd_F }{F}^2
+ h_F \ltwo{ \jump{\bE_h}\upc_F}{F}^2
\bigg\}^{\frac12}.
\]
We apply the local broken-curl, divergence-preserving Poincar\'e inequality to
$\bv:=\psi_{\vtx} \bE_h \in \bP\upb_{p+1}(\meshv)\subset \bH^1(\meshv)$ and select
$\bv\upc_\vtx:=\bE_{\vtx}$, which indeed satisfies $\bv\upc_\vtx\in \Hrotzv$
and $\bv\upc_\vtx-\bv=\bU_\vtx+\nabla\theta_\vtx - \psi_{\vtx} \bE_h\in \Hdivvz$
since $\bU_\vtx \in \Hdivvz$ by construction and $
\nabla\theta_\vtx - \psi_{\vtx} \bE_h\in \Hdivvz$ by definition of $\theta_\vtx$.
Observing that $\bdelta_{\vtx}= \bE_{\vtx} - \psi_{\vtx} \bE_h = \bv\upc_\vtx-\bv$, we infer that
\begin{align*}
\| \bdelta_{\vtx} \|_{\omvtx} & \lesssim h_\vtx \bigg\{ \|\bROT \bdelta_\vtx \|_{\omvtx}
+
\! \sum_{F\in \bFa}  \!
h_F^{-1} \ltwo{ \jump{\psi_\vtx\bE_h}\upc_F}{F}^2
\bigg\}^{\frac12} \\
&\le h_\vtx \bigg\{ \|\bROT \bdelta_\vtx \|_{\omvtx}
+
\! \sum_{F\in \Fv } \!
h_F^{-1} \ltwo{\jump{\bE_h}\upc_F}{F}^2
\bigg\}^{\frac12},
\end{align*}
where we used that $\psi_\vtx$ is a continuous function over $\omvtx$, vanishing on $\partial\omvtx$ and bounded by one.
Invoking~\eqref{eq:broken curl} to bound the first term on the right-hand side involving the broken curl and using the mesh shape-regularity gives
$$
\| \bdelta_{\vtx} \|_{\omvtx} \lesssim \bigg\{
 \! \sum_{F\in \Fv   } \!
\Big( \frac{h_F^3}{p}\Big) \ltwo{\jump{ \ROT   \bE_{h}}\upd_F }{F}^2
+
p^2 h_F \ltwo{ \jump{\bE_h}\upc_F}{F}^2
\bigg\}^{\frac12}.
$$
Factoring out $p^2$ completes the proof.

\begin{remark}[Definition of $\btheta_\vtx$]
The above proof highlights the relevance of the local problem~\eqref{def: Patch reconstruction curl free} defining the curl-free part of $\bE_\vtx$, i.e., $\nabla \theta_\vtx$.
\end{remark}

\subsection{Proof of Lemma~\ref{lem:L2 Sobolev}} \label{sec:proof_L2 Sobolev}

We want to prove that, provided Assumption~\ref{ass:Neumann} holds true,
\[
\| \bdelta_{\vtx} \|_{\omvtx} \lesssim \bigg\{
\! \sum_{F\in \Fv } \! h_F \ltwo{ \jump{\bE_h}\upc_F}{F}^2
\bigg\}^{\frac12}.
\]
We use the Aubin--Nitsche duality argument to bound $\bdelta_{\vtx}$ in the  $\bL^2$-norm on the vertex patch $\omvtx$. Specifically, we consider the following dual problem: Find $\bZ_{\vtx} \in \bXzv$ such that
\begin{subequations}
\label{eq_duality maxwell_strong}
\begin{alignat}{2}
\ROT (\ROTZ \bZ_{\vtx}) &= \bdelta_{\vtx} &\quad&\text{in $\omvtx$}, \label{eq:dual_local_PDE}
\\
\bZ_{\vtx} \CROSS \bn_{\omvtx} &= \bzero &\quad&\text{on $\partial \omvtx$}.
\end{alignat}
\end{subequations}
This problem is well-posed since $\omvtx$ has trivial topology and $\bdelta_{\vtx}$ is divergence-free. Indeed, $\DIV \bdelta_{\vtx} = \DIV (\bE_{\vtx} -  (\psi_{\vtx} \bE_{h})) =  \DIV \bU_{\vtx} + \DIV(\nabla \theta_{\vtx} -  \psi_{\vtx} \bE_{h}) = 0$ since $\bU_{\vtx}$ is divergence-free by construction and the definition \eqref{def: Patch reconstruction curl free} of $\theta_{\vtx}$ implies that $\nabla \theta_{\vtx} -  \psi_{\vtx} \bE_{h}$ is divergence-free as well.
Recalling the regularity pickup from \cite[Section 5.2]{Ciarlet16},  and \cite[Lemma 44.2]{EG_volII}, we infer that there is $\delta_\vtx' \in(0,\frac12]$ (here, $\delta_\vtx'$ can depend on $\vtx\in\vertice$) so that
$\ROTZ \bZ_{\vtx} \in \bH^{\frac12 + \delta_\vtx'}(\omvtx)$. This implies in particular that $\ROTZ \bZ_{\vtx}$ has meaningful (tangential) traces in $\bL^2(\partial K)$ for  all $K\in\meshv$. Taking the $\bL^2(\omvtx)$-inner product with $\bdelta_{\vtx}=\bU_{\vtx}+\nabla \theta_{\vtx}-\psi_{\vtx} \bE_{h}$ in \eqref{eq:dual_local_PDE}, we obtain
\begin{align*}
\ltwo{\bdelta_{\vtx}}{\omvtx}^2 &= ( \ROT (\ROTZ \bZ_{\vtx}) , \bdelta_{\vtx})_{\omvtx} \\
&=  (  \ROT (\ROTZ \bZ_{\vtx}) ,  \bU_{\vtx})_{\omvtx}
+ (\ROT (\ROTZ \bZ_{\vtx}),  \nabla \theta_{\vtx}  )_{\omvtx}
- (  \ROT (\ROTZ \bZ_{\vtx}) , \psi_{\vtx} \bE_{h})_{\omvtx}.
\end{align*}
Integrating by parts the curl operator, we infer that
\begin{align*}
\ltwo{\bdelta_{\vtx}}{\omvtx}^2 =& (  \ROTZ \bZ_{\vtx} ,  \ROTZ \bU_{\vtx})_{\omvtx}
- (   \ROTZ \bZ_{\vtx} , \bROT (\psi_{\vtx} \bE_{h}))_{\omvtx}\\
&- \sum_{K\in\meshv}(  \ROTZ \bZ_{\vtx} ,  (\psi_{\vtx} \bE_{h})\CROSS \bn_{K})_{\dK}.
\end{align*}
Re-arranging the last term on the right-hand side leads to
\[
\ltwo{\bdelta_{\vtx}}{\omvtx}^2 =  (  \ROTZ \bZ_{\vtx} ,  \bROT  (\bU_{\vtx} - (\psi_{\vtx} \bE_{h})))_{\omvtx} - \sum_{F\in \Fv} (  \ROTZ \bZ_{\vtx},\psi_{\vtx} \jump{  \bE_{h}}\upc_F)_F.
\]
Next, invoking the relation \eqref{def: Patch reconstruction} with $\bW_{\vtx}:= \bZ_{\vtx} \in \bXzv$ and the definition~\eqref{eq:def_Lh} of the lifting operator, we infer that
\begin{equation}\label{eq: error bound I}
\begin{split}
\ltwo{\bdelta_{\vtx}}{\omvtx}^2
& = ( \ROTZ \bZ_{\vtx} , \bLhpz(\psi_{\vtx} \bE_h))_{\omvtx}
- \sum_{F\in \Fv} (  \ROTZ \bZ_{\vtx},\psi_{\vtx} \jump{  \bE_{h}}\upc_F)_F\\
& = ( \bPi\upb_{\vtx} (\ROTZ \bZ_{\vtx}) , \bLhpz(\psi_{\vtx} \bE_h))_{\omvtx}
- \sum_{F\in \Fv} (  \ROTZ \bZ_{\vtx},\psi_{\vtx} \jump{  \bE_{h}}\upc_F)_F\\
& =
- \sum_{F\in \Fv} ( \avg{ \ROTZ \bZ_{\vtx} - \bPi\upb_{\vtx} (\ROTZ \bZ_{\vtx})}\upg_F ,\psi_{\vtx} \jump{  \bE_{h}}\upc_F)_F,
\end{split}
\end{equation}
where $\bPi\upb_{\vtx}$ denotes the $\bL^2$-orthogonal projection onto $\bP\upb_p(\meshv)$. Using the Cauchy--Schwarz inequality, $\|\psi_{\vtx}\|_{L^\infty(F)}=1$,
bounding the mean-values by the traces in the (one or) two cell(s) sharing $F\in\Fv$,
and invoking the mesh shape-regularity, we infer that
\begin{align*}
\ltwo{\bdelta_{\vtx}}{\omvtx}^2 &\!\le\!
\bigg\{ \!  \sum_{F\in \Fv  } \!
\Big(\frac{p}{h_F}\Big)^{2\delta} \!\ltwo{\avg{ \ROTZ \bZ_{\vtx} - \bPi\upb_{\vtx} (\ROTZ \bZ_{\vtx})}\upg_F}{F}^2 \bigg\}^{\frac12}\! \bigg\{ \!\!\sum_{F\in \Fv   }\!\! \Big(\frac{h_F}{p}\Big)^{2\delta} \!\ltwo{ \jump{ \bE_{h}}\upc_F }{F}^2 \bigg\}^{\frac12} \\
& \!\lesssim\! \bigg\{ \!\!  \sum_{K\in\meshv} \!\!
\Big(\frac{p}{h_K}\Big)^{2\delta} \ltwo{\ROTZ \bZ_{\vtx} - \bPi\upb_K (\ROTZ \bZ_{\vtx})}{\dK}^2 \bigg\}^{\frac12}
\bigg\{ \!\sum_{F\in \Fv   }\! \Big(\frac{h_F}{p}\Big)^{2\delta} \ltwo{ \jump{ \bE_{h}}\upc_F }{F}^2 \bigg\}^{\frac12},
\end{align*}
where $\delta\in(0,\frac12]$ is given by Assumption~\ref{ass:Neumann}.
Using the approximation result~\eqref{eq: L2 projection Polynomial approximation} on the $L^2$-projection followed by Lemma~\ref{lemma: regularity} (notice that indeed $\ROTZ\bZ_\vtx \in \bXdivzv$), we infer that, for all $K\in\meshv$,
\[
\Big(\frac{p}{h_K}\Big)^{\delta} \ltwo{\ROTZ \bZ_{\vtx} - \bPi\upb_K (\ROTZ \bZ_{\vtx})}{\dK}
\lesssim |\ROTZ \bZ_\vtx|_{\bH^{\frac12+\delta}(K)}
\lesssim h_{\vtx}^{\frac12 -\delta}  \ltwo{\bdelta_{\vtx} }{\omvtx}.
\]
Combining the above two bounds and since $h_{\vtx} \lesssim h_F$ for all $F\in \Fv$ and $p\ge 1$, we infer that the assertion holds true. (Notice that we just dropped the factor $p^{-2\delta}$ scaling the tangential jumps of $\bE_h$ since $\delta$ can be arbitrarily small.)

\section{Application to $hp$-a posteriori error analysis}
\label{sec: a posteriori error for dG}

In this section, we show how the $\Hrtz$-reconstruction operator constructed above can be used in residual-based $hp$-a posteriori error analysis. We focus on the symmetric interior penalty dG approximation of some simplified forms of Maxwell's equations involving the curl-curl operator and a zero-order term. Other nonconforming approximation methods can be considered as well. For simplicity, we assume that Assumption~\ref{ass:Neumann} holds true.

\subsection{Model problem}
\label{sec:model}

The model problem consists in finding $\bE \in \Hrotz$ such that
\begin{equation}
\label{eq_maxwell_strong}
\omega^2 \eps \bE+\ROT(\bnu\ROTZ \bE) = \bJ \quad\text{in $\Dom$},
\end{equation}
where $\bJ \in \Hdiv$ is the source term.
Both material properties $\eps$ and $\bnu$ can vary in $\Dom$ and
are assumed to take positive-definite, symmetric, piecewise constant matrix values
with eigenvalues uniformly bounded from above and below away from zero. Notice that, for simplicity, we are enforcing homogeneous Dirichlet boundary conditions (also called perfect electric conductor boundary conditions). The parameter $\omega$ scales as a frequency (reciprocal of a time scale) and can be either a real number (leading to the so-called definite Maxwell problem) or a pure imaginary number (leading to the so-called indefinite problem). Whenever $\omega=0$, the additional property $\DIV\bJ=0$ is prescribed. For the time being, we assume that $\omega$ is a nonzero real number and refer the reader to Remarks~\ref{rem:indefinite_apriori} and~\ref{rem:indefinite} for further insight into the other cases.

For a material property $\balpha\in\{\eps,\bnu\}$, we introduce
the inner product and associated norm weighted by $\balpha$,
leading to the notation $(\SCAL,\SCAL)_\balpha$ and $\|\SCAL\|_\balpha$.
The weak formulation of~\eqref{eq_maxwell_strong} classically amounts
to finding $\bE\in\Hrotz$ such that
\begin{equation} \label{eq:weak}
b(\bE,\bw) = (\bJ,\bw) \quad \forall \bw \in \Hrotz,
\end{equation}
with the bilinear form defined on $\Hrotz\times\Hrotz$ such that
\begin{equation}
b(\bv,\bw) \eqq  \omega^2(\bv,\bw)_\eps + (\ROTZ\bv,\ROTZ\bw)_{\bnu}.
\end{equation}

\subsection{Symmetric interior penalty dG discretization}
\label{sec:ipdg}

We employ the symmetric interior penalty dG method to discretize the weak formulation~\eqref{eq:weak} (see \cite{PeScM:02,HoPSS:05}). We assume that the mesh $\mesh$ is compatible with the domain partition on which $\eps$ and $\bnu$ are piecewise constant.
For $\balpha\in\{\eps,\bnu\}$, we set $\balpha_K \eqq \balpha|_K$ for all $K\in\mesh$, and denote by $\alpha_K^\flat$ and $\alpha_K^\sharp$ the smallest and largest eigenvalue of $\balpha_K$, respectively. It is convenient to set
\begin{equation}
\alpha_F^{\sharp}: = \max(\alpha_{K_l}^{\sharp},\alpha_{K_r}^{\sharp}) \; \forall
F := \partial K_l \cap \partial K_r \in \Fint,
\quad
\alpha_F^{\sharp}: = \alpha_{K_l}^{\sharp} \; \forall F := \dK_l \cap \partial {\Omega} \in \Fb.
\end{equation}
For all $K\in\mesh$, we also define $\alpha^\sharp_{\dK}$ so that
$\alpha^\sharp_{\dK}|_F := \alpha^\sharp_F$ for all $F\in \calFK$.
To avoid distracting technicalities, we assume that the material properties are only
mildly heterogeneous and anisotropic. Thus, we hide anisotropy ratios such as
$\alpha^\sharp_K/\alpha^\flat_K$ in the generic constants used in the error analysis.
We do the same for contrast factors such as $\alpha^\sharp_K/\alpha^\sharp_{K'}$, where
$K'$ lies in some neighborhood of $K$. We refer the reader to Remarks~\ref{rem:weights}
and~\ref{rem:hetero} for some further insight.

The discrete problem reads as follows: Find $\bE_h\in \Ppb$ (recall that $p\ge1$) such that
\begin{equation} \label{eq:disc_pb}
b_h(\bE_h,\bw_h)=(\bJ,\bw_h) \qquad \forall \bw_h\in \Ppb,
\end{equation}
with the discrete bilinear form such that, for all $\bv_h,\bw_h\in \Ppb$,
\begin{equation} \label{eq:IPDG} \begin{aligned}
b_h(\bv_h,\bw_h) := {}& \omega^2(\bv_h,\bw_h)_\eps + (\ROTh \bv_h,\ROTh\bw_h)_\bnu
+ \eta_* s_h(\bv_h,\bw_h) \\
& + \sum_{F\in\Fall} \big\{ (\avg{\bnu\ROTh\bv_h}\upg_F,\jump{\bw_h}\upc_F)_{F}
+ (\jump{\bv_h}\upc_F,\avg{\bnu\ROTh\bw_h}\upg_F)_{F}\big\}.
\end{aligned} \end{equation}
The stabilization bilinear form $s_h$ is defined as
\begin{equation} \label{eq:stab_IPDG}
s_h(\bv_h,\bw_h) \eqq \sum_{F\in\Fall} \nu_F^{\sharp} \Big( \frac{p^2}{h_F}\Big)
(\jump{\bv_h}\upc_F,\jump{\bw_h}\upc_F)_{F},
\end{equation}
and the user-dependent parameter $\eta_*>0$ is taken large enough (see below).
Using the definition~\eqref{eq:def_Lh} of the lifting operator $\bLhpz$, we infer that
\begin{equation} \label{eq:bh_lifting} \begin{aligned}
b_h(\bv_h,\bw_h) := {}& \omega^2(\bv_h,\bw_h)_\eps + (\ROTh \bv_h,\ROTh\bw_h)_\bnu
+ \eta_* s_h(\bv_h,\bw_h)  \\
& + (\bLhpz(\bw_h),\ROTh\bv_h)_\bnu
+  (\bLhpz(\bv_h),\ROTh\bw_h)_\bnu.
\end{aligned} \end{equation}

We equip the discrete space $\Ppb$ with the norm
\begin{equation} \label{eq:dGnorm}
\tnormdG{\bv_h}^2 \eqq \omega^2 \|\bv_h\|_\eps^2 + \| \ROTh \bv_h\|_\bnu^2 + s_h(\bv_h,\bv_h)
\quad \forall \bv_h\in\Ppb.
\end{equation}
Recalling~\eqref{eq:bnd_on_L}, we infer that
\begin{equation}\label{eq: bounds of Lifting}
\|\bLhpz(\bv_h)\|_\bnu^2 \leq C(\kappa_{\mesh}) s_h(\bv_h,\bv_h) \quad \forall \bv_h\in\Ppb,
\end{equation}
whence we infer, provided we take $\eta_* \geq \frac{1}{2}+2C(\kappa_{\mesh})$
with $C(\kappa_{\mesh})$ from~\eqref{eq: bounds of Lifting}, that the following coercivity property holds:
\begin{equation} \label{eq:coercivity}
b_h(\bv_h,\bv_h) \ge \frac12 \tnormdG{\bv_h}^2 \quad \forall \bv_h\in\Ppb.
\end{equation}

\begin{remark}[Weighted averages] \label{rem:weights}
Whenever the jumps of $\bnu$ across the mesh interfaces are large,
one can consider weighted ($\bnu$-dependent) averages to evaluate the
last two terms on the right-hand side of~\eqref{eq:IPDG}. In the case
of strong anisotropy, the weights depend on the normal-normal component of $\bnu$
on both sides of the interface. We refer the reader, e.g., to
\cite{ErStZ:09} for an example in the context of a scalar diffusion problems.
Accordingly, the definition~\eqref{eq:def_Lh} of the lifting operator $\bLhpz$ can be
modified by using a weighted average on the right-hand side.
\end{remark}

\begin{remark}[Indefinite case] \label{rem:indefinite_apriori}
In the indefinite case, the coercivity property \eqref{eq:coercivity} is replaced by a G{\aa}rding inequality. Assuming that $\omega$ is not a resonant frequency and invoking a duality argument \`a la Schatz then leads to a discrete inf-sup condition if the mesh size is small enough. We refer the reader to \cite[Sec.~5.4]{chaumontfrelet:hal-04589791} for more details.
\end{remark}

\subsection{Estimator and error measure}

The a posteriori error estimator is written as the sum over the mesh cells
of local error indicators $\eta_K$ for all $K\in\calT_h$. The local error indicator
consists of three pieces. The first two respectively measure the residuals of the
divergence constraint and of Maxwell's equations:
\begin{subequations}
\label{eq_def_eta}
\begin{equation}
\etadK^2
:=
\frac{1}{\epsilon^\sharp_K}
\bigg \{ \frac{1}{\omega^{2}}
\Big( \frac{h_K}{p} \Big)^2 \|\DIV (\bJ-\omega^2\eps\bE_h)\|_K^2
+ \omega^2
\Big(\frac{h_K}{p}\Big) \|\jump{\eps\bE_h}_{\partial K}\upd\|_{\partial K \setminus \partial \Omega}^2
\bigg \}, \\
\end{equation}
and
\begin{equation}
\begin{split}
\etacK^2
:={}&
\frac{1}{\nu^\sharp_K} \bigg \{
\Big( \frac{h_K}{p}\Big)^2 \|\bJ-\omega^2 \eps \bE_h-\ROT (\bnu \ROT\bE_h)\|_K^2\\
&+
\Big(\frac{h_K}{p}\Big) \|\jump{\bnu  \bROT \bE_h}_{\partial K}\upc\|_{\partial K \setminus \partial \Omega}^2
+ \nu_{\dK}^{\sharp} \Big(\frac{p^2}{h_K}\Big) \|\jump{\bE_h}_{\partial K}\upc\|_{\partial K}^2
\bigg \}.
\end{split}
\end{equation}
The last part of the estimator is related to the nonconformity of the discrete
field $\bE_h$ and reads
\begin{equation}
\label{eq_def_etajK}
\etajK^2
:= \nu_{\dK}^\sharp \Big(\frac{h_K}{p}\Big) \ltwo{\jumpK{ \bROT\bE_{h}}\upd }{\dK}^2
+
\bigg\{ {\omega^2 \epsilon_{\dK}^\sharp h_K} + \nu_{\dK}^\sharp\Big(\frac{p^2}{h_K}\Big)\bigg\} \ltwo{ \jumpK{\bE_h}\upc}{\dK}^2.
\end{equation}
\end{subequations}
Notice that this latter estimator is suitable under Assumption~\ref{ass:Neumann}.
It is convenient to introduce the following shorthand notation:
\begin{equation} \label{eq:def_eta}
\etaK^2 := \etadK^2+\etacK^2+\etajK^2,
\qquad
\eta_{\bullet}^2 := \sum_{K \in \calT_h} \eta_{\bullet,K}^2,
\qquad
\eta^2 := \sum_{K \in \calT_h} \etaK^2,
\end{equation}
with $\bullet\in\{\operatorname{div},\operatorname{curl},\operatorname{nc}\}$.

The approximation error is defined as
\begin{equation} \label{eq:def_bVsh}
\be:=\bE-\bE_h \in \bVsh \eqq \Hrotz + \bP\upb_p(\calT_h).
\end{equation}
Although the sum defining $\bVsh$ is not direct, any field $\bv_h\in \Hrotz\cap \Ppb$
satisfies $\jump{\bv_h}\upc_F=\bzero$ for all $F\in\Fall$. It is therefore legitimate
to extend the jump operator to $\bVsh$ by setting $\jump{\bv}\upc_F:=\jump{\bv_h}\upc_F$
for all $\bv=\tilde \bv+\bv_h \in \bVsh$. Similarly, the lifting operator $\bLhpz$ can be extended
to $\bVsh$ by setting $\bLhpz(\bv):=\bLhpz(\bv_h)$. The discrete bilinear form $b_h$
is then extended to $\bVsh\times\bVsh$ by using the expression~\eqref{eq:bh_lifting},
i.e., we set, for all $\bv,\bw\in \bVsh$,
\begin{equation} \label{eq:IPDG extended} \begin{aligned}
\bs(\bv,\bw) := {}& \omega^2(\bv,\bw)_\eps + (\ROTh \bv,\ROTh\bw)_\bnu
+ \eta_* s_h(\bv,\bw)  \\
& + (\bLhpz(\bw),\ROTh\bv)_\bnu
+  (\bLhpz(\bv),\ROTh\bw)_\bnu,
\end{aligned} \end{equation}
keeping in mind that $\ROTh\bv:=\ROTZ\bv$ whenever $\bv\in\Hrotz$.

For any subset $\calT\subset \mesh$, we define the error measure
\begin{equation} \label{eq:error_meas}
\tnorm{\be}_{\sharp,\calT}^2
:= \sum_{K \in \calT}
\bigg \{
\omega^2 \|\be\|_{\eps,K}^2
+
\|\ROT \be\|_{\bnu,K}^2 + \nu_{\dK}^{\sharp}
\Big( \frac{p^2}{h_K} \Big) \|\jump{\be}_{\partial K}\upc\|_{\partial K}^2
\bigg \},
\end{equation}
and omit the subscript $\calT$ whenever $\calT = \mesh$.
When restricted to $\Ppb$, $\tnorm{\bullet}_{\sharp}$ is equivalent to the discrete
dG norm $\tnormdG{\bullet}$ defined in~\eqref{eq:dGnorm} up to a rewriting of the
jump contribution as a sum over the mesh cell boundaries rather than over the mesh faces.

\begin{remark}[Stabilization weight]
We observe that the error measure defined in~\eqref{eq:error_meas} is independent
 of the stabilization parameter $\eta_*$ introduced in the dG scheme.
\end{remark}

\begin{remark}[Heterogeneous materials] \label{rem:hetero}
Whenever one wishes to account for the heterogeneity of the material properties,
some weights in the error indicators need to be adapted. For the error indicators
defined in~\eqref{eq_def_eta}, one needs to consider a contrast factor involving
three layers of vertex-based neighbors around $K$, as detailed in
\cite{chaumontfrelet:hal-04589791}. For the nonconformity indicator defined in
\eqref{eq_def_etajK}, one layer of neighbors is sufficient. We also notice that
it is possible to modify Definition~\ref{def: Patchwise and global potential reconstruction}
so as to reconstruct on each vertex patch an $\Hrtz$-conforming field by using
weighted $L^2$-inner products. However, the estimates satisfied by this modified
reconstruction in the broken curl and $\bL^2$-norms will depend anyway on the contrast
factors associated with the underlying weights.
\end{remark}

\subsection{$hp$-a posteriori error analysis}

We decompose the approximation error into two components as follows:
\begin{equation}\label{def: error splitting}
\be=(\bE-\bE_c)+(\bE_c- \bE_h)=:\be_c+\be_{nc}, \qquad \bE_c:=\opRec(\bE_h) \in \Hrotz.
\end{equation}
We call $\be_c$ the conforming error and $\be_{nc}$ the nonconforming error.
Notice that we use the $\Hrtz$-reconstruction operator from Section~\ref{sec:main results}
to define the two error components. Actually, the precise definition of $\bE_c$ is
irrelevant for bounding $\be_c$ (we only use that $\be_c\in \Hrotz$), and is only relevant
for bounding $\be_{nc}$.

\begin{lemma}[Residual]
\label{lemma_pde_residual}
Let $\etadc^2 := \etad^2+\etac^2$. The following holds:
\begin{equation}
\label{eq_residual_Hcurl}
|\bs(\be,\be_c)|
\leq
C(\kappa_{\calT_h}) \etadc \tnorms{\be_c}.
\end{equation}
\end{lemma}

\begin{proof}
The proof essentially follows the arguments from the proof of \cite[Lem.~6.2]{chaumontfrelet:hal-04589791}, up to minor adaptations (the proof therein tracks the dependency of the constants on
the heterogeneity of the material properties, and deals with the indefinite Maxwell problem).
\end{proof}

\begin{lemma}[Bound on $\be_c$]
\label{conforming error in energy norm}
The following holds:
\begin{equation}
\label{eq:comforming error in energy norm}
\tnorms{\be_{c}}
\leq
C(\kappa_{\calT_h}) (\etadc+\tnorms{\be_{nc}}).
\end{equation}
\end{lemma}

\begin{proof}
Since $\be_c\in\Hrotz$, we have $s_h(\be_c,\be_c)=0$ and $\bLhpz(\be_c)=\bzero$, so that
\begin{align*}
\tnorms{\be_c}^2 &= \omega^2 \|\be_c\|_\eps^2 + \|\ROTZ \be_c\|_\bnu^2 \\
&= \bs(\be_c, \be_c) = \bs(\be, \be_c) - \bs(\be_{nc}, \be_c) =: T_1 + T_2.
\end{align*}
Applying the bound~\eqref{eq_residual_Hcurl} from Lemma~\ref{lemma_pde_residual}, we get
\begin{equation*}
|T_1| \lesssim \etadc \tnorms{\be_c}.
\end{equation*}
Moreover, we observe that
\[
T_2 = \bs(\be_{nc},\be_c) = \omega^2(\be_{nc},\be_c)_\eps + (\ROTh \be_{nc},\ROTZ\be_c)_\bnu
- (\bLhpz(\bE_h),\ROTZ\be_c)_{\bnu}.
\]
Invoking the Cauchy--Schwarz inequality and the bound~\eqref{eq:bnd_on_L} on the discrete lifting operator, we infer that
\[
|T_2| \leq 2 \Big\{\omega^2\|\be_{nc}\|_\eps^2 + \|\ROTh\be_{nc}\|_\bnu^2 +\|\bLhpz(\bE_h)\|_\bnu^2 \Big\}^{\frac12} \tnorms{\be_{c}} \lesssim \tnorms{\be_{nc}} \tnorms{\be_{c}}.
\]
Combining the above bounds completes the proof.
\end{proof}

\begin{lemma}[Bound on $\be_{nc}$]
\label{nonconforming error in energy norm}
Under Assumption~\ref{ass:Neumann}, the following holds:
\begin{equation}
\label{eq:noncomforming error in energy norm}
\tnorms{\be_{nc}}
\leq
C(\kappa_{\calT_h}) \etaj.
\end{equation}
\end{lemma}

\begin{proof}
Recall that, for all $K\in\mesh$,
$\be_{nc}|_K = \sum_{\vtx \in \verticeK} \bdelta_{\vtx}|_K$ with
$\bdelta_{\vtx}:= \bE_{\vtx} - \psi_{\vtx} \bE_h$ (see~\eqref{eq:encvtx}).
Reasoning as in the proof of
Theorem~\ref{theorem: nonconforming error total}, this gives
\[
\|\ROT \be_{nc}\|_{\bnu,K}^2 \le \nu_K \sum_{\vtx\in \verticeK} \|\ROTh \bdelta_\vtx\|_{\omvtx}^2.
\]
Invoking the bound on $\|\ROTh \bdelta_\vtx\|_{\omvtx}$
from Lemma~\ref{lem:broken curl},
and hiding the contrast factor on $\nu$ in the generic constants, we infer that
\[
\|\ROT \be_{nc}\|_{\bnu,K}^2 \lesssim \sum_{\vtx\in \verticeK} \sum_{F\in \Fv   } \!
 \bigg\{ \nu^\sharp_F \Big(\frac{h_F}{p}\Big) \ltwo{\jump{ \bROT \bE_{h}}\upd_F }{F}^2
+ \nu^\sharp_F \Big(\frac{  p^2}{h_F} \Big) \ltwo{ \jump{\bE_h}\upc_F}{F}^2 \bigg\}.
\]
Reasoning similarly to bound $\omega^2\|\be_{nc}\|_{\eps,K}^2$, but this time invoking the bound from Lemma~\ref{lem:L2 Sobolev}, we infer that
\[
\omega^2\|\be_{nc}\|_{\eps,K}^2 \lesssim \sum_{\vtx\in \verticeK}
\sum_{F\in \Fv } \! \omega^2 \epsilon^\sharp_F h_F \ltwo{ \jump{\bE_h}\upc_F}{F}^2.
\]
Combining the above two bounds together, re-arranging the summations over the mesh cells,
invoking again the mild heterogeneity of the material properties, and observing
that $\jumpK{\be_{nc}}\upc=-\jumpK{\bE_h}\upc$ for all $K\in\mesh$, leads to the expected
bound on $\tnorms{\be_{nc}}$.
\end{proof}

\begin{theorem}[Global upper bound (reliability)] \label{th:reliability}
Under Assumption~\ref{ass:Neumann}, the following holds:
\begin{equation}
\label{eq_reliability}
\tnorms{\be}
\leq
C(\kappa_{\calT_h})
\eta.
\end{equation}
\end{theorem}

\begin{proof}
Invoke the triangle inequality together with Lemmas~\ref{conforming error in energy norm} and~\ref{nonconforming error in energy norm}, and recall the definition of $\eta$ from~\eqref{eq:def_eta}.
\end{proof}

Finally, we derive local efficiency estimates.

\begin{proposition}[Local lower bound (local efficiency)] \label{prop:efficiency}
For all $K \in \calT_h$, we have
\begin{equation}
\eta_K
\leq
C(\kappa_{\calT_h}) p^{\frac32}\left \{
\bigg (
1 + \Big(\frac{\omega^2 \epsilon_K^\sharp}{\nu_K^\sharp}\Big)^{\frac12} \Big(\frac{h_K}{p}\Big)
\bigg )
\tnorm{\bE-\bE_h}_{\sharp,\Kupf}
+
\operatorname{osc}_{\Kupf}
\right \},
\end{equation}
with the data oscillation term
\begin{equation} \label{eq:data_osc}
\operatorname{osc}_{\Kupf}^2
\!:=\!\!\!
\sum_{K' \in \Kupf}
\min_{\bJ_{h} \in \Ppb} \!\!
\left \{
\frac{1}{\nu_{K'}^\sharp} \Big(\frac{h_{K'}}{p} \Big)^2 \|\bJ-\bJ_{h}\|_{K'}^2
\! + \!
\frac{1}{\omega^2 \epsilon_{K'}^\sharp} \Big(\frac{h_{K'}}{p}\Big)^2 \|\DIV (\bJ-\bJ_{h})\|_{K'}^2
\right \},
\end{equation}
and $\Kupf$ is the collection of the mesh cells sharing a face with $K$.
\end{proposition}

\begin{proof}
Local bounds on all terms composing the error estimator $\eta_K$ can be found in \cite[Theorem 6.5]{chaumontfrelet:hal-04589791} except for the term involving $\jump{ \bROT \bE_{h}}_F\upd$. Invoking the $H^1$- to $L^2$-norm inverse inequality from \cite[Theorem 4.76]{schwab}, we infer that
\begin{equation*}
\nu_{\dK}^\sharp\Big(\frac{h_K}{p}\Big) \ltwo{\jumpK{ \bROT \bE_{h}}\upd }{\dK}^2
\lesssim \nu_{\dK}^\sharp \Big(\frac{p^3}{h_K}\Big) \ltwo{\jumpK{\bE_{h}}\upc }{\dK}^2
\lesssim  p \tnorm{\bE-\bE_h}_{\sharp,\Kupf}.
\end{equation*}
This completes the proof.
\end{proof}

\begin{remark}[Indefinite case] \label{rem:indefinite}
Whenever $\omega=0$, the above $hp$-a posteriori error analysis can be applied with the additional simplification of discarding the $\bL^2$-norm estimate, so that Assumption~\ref{ass:Neumann} is no longer relevant. (Notice also that the last term on the right-hand side of~\eqref{eq:data_osc} is no longer relevant since $\bJ$ is divergence-free). In the time-harmonic regime with a nonzero, pure imaginary $\omega$, the techniques in \cite{chaumontfrelet:hal-04589791} can be combined with the present $\Hrtz$-conforming reconstruction.
\end{remark}

\bibliographystyle{siam}
\bibliography{biblio}

@article {DarDurPraVam1996,
    AUTHOR = {Dari, E. and Duran, R. and Padra, C. and Vampa, V.},
     TITLE = {A posteriori error estimators for nonconforming finite element
              methods},
   JOURNAL = {RAIRO Mod\'el. Math. Anal. Num\'er.},
  FJOURNAL = {RAIRO Mod\'elisation Math\'ematique et Analyse Num\'erique},
    VOLUME = {30},
      YEAR = {1996},
    NUMBER = {4},
     PAGES = {385--400},
       optDOI = {10.1051/m2an/1996300403851},
       optURL = {https://doi.org/10.1051/m2an/1996300403851},
}

@unpublished{chaumontfrelet:hal-05204325,
  TITLE = {{Computable Poincar{\'e}--Friedrichs constants for the $L^p$ de Rham complex over convex domains and domains with shellable triangulations}},
  AUTHOR = {Chaumont-Frelet, T. and Licht, M. W. and Vohral{\'i}k, M.},
  NOTE = {Preprint, \texttt{https://inria.hal.science/hal-05204325}},
  YEAR = {2025},
  HAL_ID = {hal-05204325},
  HAL_VERSION = {v1},
}

@article {ErnVohralik2020,
    AUTHOR = {Ern, A. and Vohral\'ik, M.},
     TITLE = {Stable broken {$H^1$} and {$H({\rm div})$} polynomial
              extensions for polynomial-degree-robust potential and flux
              reconstruction in three space dimensions},
   JOURNAL = {Math. Comp.},
  FJOURNAL = {Mathematics of Computation},
    VOLUME = {89},
      YEAR = {2020},
    NUMBER = {322},
     PAGES = {551--594},
      ISSN = {0025-5718,1088-6842},
   MRCLASS = {65N15 (65N30 76M10)},
  MRNUMBER = {4044442},
MRREVIEWER = {Riccardo\ Sacco},
       optDOI = {10.1090/mcom/3482},
       optURL = {https://doi.org/10.1090/mcom/3482},
}

@article {Schatz1974,
    AUTHOR = {Schatz, A. H.},
     TITLE = {An observation concerning {R}itz-{G}alerkin methods with
              indefinite bilinear forms},
   JOURNAL = {Math. Comp.},
  FJOURNAL = {Mathematics of Computation},
    VOLUME = {28},
      YEAR = {1974},
     PAGES = {959--962},
      ISSN = {0025-5718,1088-6842},
   MRCLASS = {65N30 (35JXX)},
  MRNUMBER = {373326},
MRREVIEWER = {J.\ W.\ McLaurin},
       DOI = {10.2307/2005357},
       URL = {https://doi.org/10.2307/2005357},
}

@article {CarBarJan2002,
    AUTHOR = {Carstensen, C. and Bartels, S. and Jansche, S.},
     TITLE = {A posteriori error estimates for nonconforming finite element
              methods},
   JOURNAL = {Numer. Math.},
  FJOURNAL = {Numerische Mathematik},
    VOLUME = {92},
      YEAR = {2002},
    NUMBER = {2},
     PAGES = {233--256},
       optDOI = {10.1007/s002110100378},
       optURL = {https://doi.org/10.1007/s002110100378},
}

@article {BecHanLar2003,
    AUTHOR = {Becker, R. and Hansbo, P. and Larson, M. G.},
     TITLE = {Energy norm a posteriori error estimation for discontinuous
              {G}alerkin methods},
   JOURNAL = {Comput. Methods Appl. Mech. Engrg.},
  FJOURNAL = {Computer Methods in Applied Mechanics and Engineering},
    VOLUME = {192},
      YEAR = {2003},
    NUMBER = {5-6},
     PAGES = {723--733},
       optDOI = {10.1016/S0045-7825(02)00593-5},
       optURL = {https://doi.org/10.1016/S0045-7825(02)00593-5},
}

@article {Ainsworth:07,
    AUTHOR = {Ainsworth, M.},
     TITLE = {A posteriori error estimation for discontinuous {G}alerkin
              finite element approximation},
   JOURNAL = {SIAM J. Numer. Anal.},
  FJOURNAL = {SIAM Journal on Numerical Analysis},
    VOLUME = {45},
      YEAR = {2007},
    NUMBER = {4},
     PAGES = {1777--1798},
       optDOI = {10.1137/060665993},
       optURL = {https://doi.org/10.1137/060665993},
}

@article {BraSc:08,
    AUTHOR = {Braess, D. and Sch{\"o}berl, J.},
     TITLE = {Equilibrated residual error estimator for edge elements},
   JOURNAL = {Math. Comp.},
  FJOURNAL = {Mathematics of Computation},
    VOLUME = {77},
      YEAR = {2008},
    NUMBER = {262},
     PAGES = {651--672},
}

@article {ErnVo:15,
    AUTHOR = {Ern, A. and Vohral{\'{\i}}k, M.},
     TITLE = {Polynomial-degree-robust a posteriori estimates in a unified
              setting for conforming, nonconforming, discontinuous
              {G}alerkin, and mixed discretizations},
   JOURNAL = {SIAM J. Numer. Anal.},
  FJOURNAL = {SIAM Journal on Numerical Analysis},
    VOLUME = {53},
      YEAR = {2015},
    NUMBER = {2},
     PAGES = {1058--1081},
}

@article {ErnVo:20,
    AUTHOR = {Ern, A. and Vohral{\'{\i}}k, M.},
     TITLE = {Stable broken {$H^1$} and {$H({\rm div})$} polynomial
              extensions for polynomial-degree-robust potential and flux
              reconstruction in three space dimensions},
   JOURNAL = {Math. Comp.},
  FJOURNAL = {Mathematics of Computation},
    VOLUME = {89},
      YEAR = {2020},
    NUMBER = {322},
     PAGES = {551--594},
       optDOI = {10.1090/mcom/3482},
       optURL = {https://doi.org/10.1090/mcom/3482},
}

@article {ChaEV:21,
    AUTHOR = {Chaumont-Frelet, T. and Ern, A. and Vohral\'ik, M.},
     TITLE = {Stable broken {$\bold H(\bold{curl})$} polynomial extensions
              and {$p$}-robust a posteriori error estimates by broken
              patchwise equilibration for the curl-curl problem},
   JOURNAL = {Math. Comp.},
  FJOURNAL = {Mathematics of Computation},
    VOLUME = {91},
      YEAR = {2021},
    NUMBER = {333},
     PAGES = {37--74},
       optDOI = {10.1090/mcom/3673},
       optURL = {https://doi.org/10.1090/mcom/3673},
}

@article {Chaum:23,
    AUTHOR = {Chaumont-Frelet, T.},
     TITLE = {A simple equilibration procedure leading to
              polynomial-degree-robust a posteriori error estimators for the
              curl-curl problem},
   JOURNAL = {Math. Comp.},
  FJOURNAL = {Mathematics of Computation},
    VOLUME = {92},
      YEAR = {2023},
    NUMBER = {344},
     PAGES = {2413--2437},
       optDOI = {10.1090/mcom/3817},
       optURL = {https://doi.org/10.1090/mcom/3817},
}

@article {ChaVo:24,
    AUTHOR = {Chaumont-Frelet, T. and Vohral\'ik, M.},
     TITLE = {A stable local commuting projector and optimal {$hp$}
              approximation estimates in {$H({\rm curl})$}},
   JOURNAL = {Numer. Math.},
  FJOURNAL = {Numerische Mathematik},
    VOLUME = {156},
      YEAR = {2024},
    NUMBER = {6},
     PAGES = {2293--2342},
       optDOI = {10.1007/s00211-024-01431-w},
       optURL = {https://doi.org/10.1007/s00211-024-01431-w},
}

@article {ChaVo:23,
    AUTHOR = {Chaumont-Frelet, T. and Vohral\'ik, M.},
     TITLE = {{$p$}-robust equilibrated flux reconstruction in {$H({\rm
              curl})$} based on local minimizations: application to a
              posteriori analysis of the curl-curl problem},
   JOURNAL = {SIAM J. Numer. Anal.},
  FJOURNAL = {SIAM Journal on Numerical Analysis},
    VOLUME = {61},
      YEAR = {2023},
    NUMBER = {4},
     PAGES = {1783--1818},
       optDOI = {10.1137/21M141909X},
       optURL = {https://doi.org/10.1137/21M141909X},
}

@Article{CamSo:16,
  author = 	 {Campos Pinto, M. and Sonnendr{\"u}cker, E.},
  title = 	 {Gauss-compatible {G}alerkin schemes for time-dependent {M}axwell equations},
  journal = 	 {Math. Comp.},
  year = 	 {2016},
  volume = 	 {302},
  OPTnumber = 	 {},
  pages = 	 {2651--2685},
}

@Article{Oswal:93,
  author = {Oswald, P.},
  title = {On a {BPX}-preconditioner for {$P_1$} elements},
  journal = {Computing},
  year = {1993},
  volume = {51},
  pages = {125--133},
}

@InProceedings{Brenn:93,
  author = 	 {Brenner, S. C.},
  title = 	 {Two-level additive {S}chwarz preconditioners for nonconforming finite elements},
  booktitle = {Domain decomposition methods in scientific and engineering computing},
  year = 	 {1993},
  editor = 	 {Keyes, D. E. and Xue, J.},
  publisher = {AMS},
  note = 	 {Proceedings of the 7th International Conference on Domain Decomposition Methods. \texttt{http://www.ddm.org/DD07/index-neu.htm}},
}

@article {Amrouche98,
    AUTHOR = {Amrouche, C. and Bernardi, C. and Dauge, M. and Girault, V.},
     TITLE = {Vector potentials in three-dimensional non-smooth domains},
   JOURNAL = {Math. Methods Appl. Sci.},
  FJOURNAL = {Mathematical Methods in the Applied Sciences},
    VOLUME = {21},
      YEAR = {1998},
    NUMBER = {9},
     PAGES = {823--864},
      ISSN = {0170-4214,1099-1476},
   MRCLASS = {35J25 (35Q30 65N30 76D07)},
  MRNUMBER = {1626990},
MRREVIEWER = {Juha\ H.\ Videman},
       DOI =
              {10.1002/(SICI)1099-1476(199806)21:9<823::AID-MMA976>3.0.CO;2-B},
       URL =
              {https://doi.org/10.1002/(SICI)1099-1476(199806)21:9<823::AID-MMA976>3.0.CO;2-B},
}

@article {houstonpersch04,
    AUTHOR = {Houston, P. and Perugia, I. and Sch\"otzau, D.},
     TITLE = {Energy norm a posteriori error estimation for mixed
              discontinuous {G}alerkin approximations of the {M}axwell
              operator},
   JOURNAL = {Comput. Methods Appl. Mech. Engrg.},
  FJOURNAL = {Computer Methods in Applied Mechanics and Engineering},
    VOLUME = {194},
      YEAR = {2005},
    NUMBER = {2-5},
     PAGES = {499--510},
      ISSN = {0045-7825,1879-2138},
   MRCLASS = {78M10 (65N15 65N30)},
  MRNUMBER = {2105178},
       DOI = {10.1016/j.cma.2004.02.025},
       URL = {https://doi.org/10.1016/j.cma.2004.02.025},
}

@book{schwab,
	author={C. Schwab},
	title={$p$-- and $hp$--{F}inite element methods: Theory and
	applications in solid and fluid mechanics},

	year=1998,
	Publisher={Oxford University Press: Numerical mathematics and
	scientific computation}}

@article{cangiani2023aposteriori,
   AUTHOR = {Cangiani, A. and Dong, Z. and Georgoulis, E. H.},
      TITLE = {{A posteriori error estimates for discontinuous Galerkin methods on polygonal and polyhedral meshes}},
   JOURNAL = {SIAM J. Numer. Anal.},
  FJOURNAL = {SIAM Journal on Numerical Analysis},
    VOLUME = {61},
      YEAR = {2023},
    NUMBER = {5},
     PAGES = {2352--2380},
      ISSN = {0036-1429},
   MRCLASS = {65N15 (65N30 74K20)},
       optDOI = {10.1137/22M1516701},
       optURL = {https://doi.org/10.1137/22M1516701},
}

@unpublished{chaumontfrelet:hal-04589791,
  TITLE = {{A priori and a posteriori analysis of the discontinuous Galerkin approximation of the time-harmonic Maxwell's equations under minimal regularity assumptions}},
  AUTHOR = {Chaumont-Frelet, T. and Ern, A.},
  URL = {https://hal.science/hal-04589791},
  NOTE = {Accepted in Math. Comp.},
  YEAR = {2025},
  DOI = {10.48550/arXiv.2412.11798},
  HAL_ID = {hal-04589791},
  HAL_VERSION = {v2},
}

@article{warburton2003constants,
	title={On the constants in $hp$-finite element trace inverse inequalities},
	author={Warburton, T. and Hesthaven, J. S.},
	journal={Comput. Methods Appl. Mech. Engrg.},
	volume={192},
	number={25},
	pages={2765--2773},
	year={2003},
	publisher={North-Holland}
}

@article {Melenkurzer14,
    AUTHOR = {Melenk, J. M. and Wurzer, T.},
     TITLE = {On the stability of the boundary trace of the polynomial
              {$L^2$}-projection on triangles and tetrahedra},
   JOURNAL = {Comput. Math. Appl.},
  FJOURNAL = {Computers \& Mathematics with Applications. An International
              Journal},
    VOLUME = {67},
      YEAR = {2014},
    NUMBER = {4},
     PAGES = {944--965},
      ISSN = {0898-1221,1873-7668},
   MRCLASS = {65D05 (41A10 41A25 41A63 46G25 65N30)},
  MRNUMBER = {3163888},
MRREVIEWER = {A.\ Bultheel},
       optDOI = {10.1016/j.camwa.2013.12.016},
       optURL = {https://doi.org/10.1016/j.camwa.2013.12.016},
}

@article {Ciarlet16,
    AUTHOR = {Ciarlet, Jr., P.},
     TITLE = {On the approximation of electromagnetic fields by edge finite
              elements. {P}art 1: {S}harp interpolation results for
              low-regularity fields},
   JOURNAL = {Comput. Math. Appl.},
  FJOURNAL = {Computers \& Mathematics with Applications. An International
              Journal},
    VOLUME = {71},
      YEAR = {2016},
    NUMBER = {1},
     PAGES = {85--104},
      ISSN = {0898-1221,1873-7668},
   MRCLASS = {65N30 (65N15 78A25)},
  MRNUMBER = {3441182},
MRREVIEWER = {Agustin\ Mart\'in},
       DOI = {10.1016/j.camwa.2015.10.020},
       URL = {https://doi.org/10.1016/j.camwa.2015.10.020},
}

@Techreport{HP19_822,
  author = {R. Hiptmair and C. Pechstein},
  title = {Regular Decompositions of Vector Fields - Continuous, Discrete, and Structure-Preserving},
  institution = {Seminar for Applied Mathematics, ETH Z{\"u}rich},
  number = {2019-18},
  address = {Switzerland},
  optURL = {https://www.sam.math.ethz.ch/sam_reports/reports_final/reports2019/2019-18.pdf },
  year = {2019}
}

@article {GuzSal21,
    AUTHOR = {Guzm\'{a}n, J. and Salgado, A. J.},
     TITLE = {Estimation of the continuity constants for {B}ogovski\u{\i}
              and regularized {P}oincar\'{e} integral operators},
   JOURNAL = {J. Math. Anal. Appl.},
  FJOURNAL = {Journal of Mathematical Analysis and Applications},
    VOLUME = {502},
      YEAR = {2021},
    NUMBER = {1},
     PAGES = {Paper No. 125246, 36},
      ISSN = {0022-247X,1096-0813},
   MRCLASS = {58A10 (35R09 42B20)},
  MRNUMBER = {4248471},
       optDOI = {10.1016/j.jmaa.2021.125246},
       optURL = {https://doi.org/10.1016/j.jmaa.2021.125246},
}

@article {ArBCM:01,
    AUTHOR = {Arnold, D. N. and Brezzi, F. and Cockburn, B. and Marini, L. D.},
     TITLE = {Unified analysis of discontinuous {G}alerkin methods for
              elliptic problems},
   JOURNAL = {SIAM J. Numer. Anal.},
    VOLUME = {39},
      YEAR = {2001/02},
    NUMBER = {5},
     PAGES = {1749--1779},
}

@ARTICLE{BurEr:07,
  author = {Burman, E. and Ern, A.},
  title = {Continuous interior penalty {$hp$}-finite element methods for advection
	and advection-diffusion equations},
  journal = {Math. Comp.},
  year = {2007},
  volume = {76},
  pages = {1119--1140},
  number = {259},
  fjournal = {Mathematics of Computation},
 OPTISSN = {0025-5718}
}

@Book{DiPEr:12,
  author     = {Di Pietro, D. A. and Ern, A.},
  title      = {Mathematical aspects of discontinuous {G}alerkin methods},
  publisher  = {Springer-Verlag},
  address    = {Berlin},
  series     = {Mathematics \& Applications},
  volume     = {69},
  year       = {2012}
}

@Article{DiPEL:14,
	author	= {Di Pietro, D. A. and Ern, A. and Lemaire, S.},
	title		= {An arbitrary-order and compact-stencil discretization of
	diffusion on general meshes based on local reconstruction
	operators},
	journal	= {Comput. Meth. Appl. Math.},
	volume	= {14},
	number	= {4},
	pages		= {461--472},
	year		= {2014}
}

@unpublished{DongErn:24,
  TITLE = {{$hp$-error analysis of mixed-order hybrid high-order methods for elliptic problems on simplicial meshes}},
  AUTHOR = {Dong, Z. and Ern, A.},
  note = {\texttt{https://inria.hal.science/hal-04720237}},
  YEAR = {2025},
  optMONTH = Jul,
  optKEYWORDS = {A posteriori error estimate ; adaptive algorithms ; hybrid high-order method ; $hp$-error estimate},
  optPDF = {https://inria.hal.science/hal-04720237v2/file/aposteriori_HHO_rev.pdf},
  optHAL_ID = {hal-04720237},
  optHAL_VERSION = {v2},
}

@article {ErnGu:17_quasi,
    AUTHOR = {Ern, A. and Guermond, J.-L.},
     TITLE = {Finite element quasi-interpolation and best approximation},
   JOURNAL = {ESAIM Math. Model. Numer. Anal.},
  FJOURNAL = {ESAIM. Mathematical Modelling and Numerical Analysis},
    VOLUME = {51},
      YEAR = {2017},
    NUMBER = {4},
     PAGES = {1367--1385},
}

@book {EG_volI,
    AUTHOR = {Ern, A. and Guermond, J.-L.},
     TITLE = {Finite Elements {I}: {A}pproximation and Interpolation},
    VOLUME = {72},
    SERIES = {Texts in Applied Mathematics},
 PUBLISHER = {Springer Nature},
   ADDRESS = {Cham, Switzerland},
      YEAR = {2021},
     optPAGES = {},
}

@book {EG_volII,
    AUTHOR = {Ern, A. and Guermond, J.-L.},
     TITLE = {Finite elements. {II.} {G}alerkin approximation, elliptic and mixed PDEs},
    SERIES = {Texts in Applied Mathematics},
    VOLUME = {73},
 PUBLISHER = {Springer, Cham},
      YEAR = {2021},
     PAGES = {492},
       optDOI = {10.1007/978-3-030-56923-5},
       optURL = {https://doi.org/10.1007/978-3-030-56923-5},
}

@ARTICLE{ErStZ:09,
  author = {Ern, A. and Stephansen, A.~F. and Zunino, P.},
  title = {A discontinuous {G}alerkin method with weighted averages for advection-diffusion
	equations with locally small and anisotropic diffusivity},
  journal = {IMA J. Numer. Anal.},
  fjournal = {IMA Journal of Numerical Analysis},
  year = {2009},
  volume = {29},
  pages = {235--256},
  number = {2},
}

@article {HoPSS:05,
    AUTHOR = {Houston, P. and Perugia, I. and Schneebeli, A. and Sch\"{o}tzau, D.},
     TITLE = {Interior penalty method for the indefinite time-harmonic
              {M}axwell equations},
   JOURNAL = {Numer. Math.},
  FJOURNAL = {Numerische Mathematik},
    VOLUME = {100},
      YEAR = {2005},
    NUMBER = {3},
     PAGES = {485--518},
       optDOI = {10.1007/s00211-005-0604-7},
       optURL = {https://doi.org/10.1007/s00211-005-0604-7},
}

@article {LeePre:79,
    AUTHOR = {Lee, D. T. and Preparata, F. P.},
     TITLE = {An optimal algorithm for finding the kernel of a polygon},
   JOURNAL = {J. Assoc. Comput. Mach.},
  FJOURNAL = {Journal of the Association for Computing Machinery},
    VOLUME = {26},
      YEAR = {1979},
    NUMBER = {3},
     PAGES = {415--421},
       optDOI = {10.1145/322139.322142},
       optURL = {https://doi-org.extranet.enpc.fr/10.1145/322139.322142},
}

@article {Melenk:05,
    AUTHOR = {Melenk, J. M.},
     TITLE = {{$hp$}-interpolation of nonsmooth functions and an application
              to {$hp$}-a posteriori error estimation},
   JOURNAL = {SIAM J. Numer. Anal.},
  FJOURNAL = {SIAM Journal on Numerical Analysis},
    VOLUME = {43},
      YEAR = {2005},
    NUMBER = {1},
     PAGES = {127--155},
       optDOI = {10.1137/S0036142903432930},
       optURL = {https://doi.org/10.1137/S0036142903432930},
}

@article {KarMe:15,
    AUTHOR = {Karkulik, M. and Melenk, J. M.},
     TITLE = {Local high-order regularization and applications to
              {$hp$}-methods},
   JOURNAL = {Comput. Math. Appl.},
  FJOURNAL = {Computers \& Mathematics with Applications. An International
              Journal},
    VOLUME = {70},
      YEAR = {2015},
    NUMBER = {7},
     PAGES = {1606--1639},
       optDOI = {10.1016/j.camwa.2015.06.026},
       optURL = {https://doi.org/10.1016/j.camwa.2015.06.026},
}

@article {PeScM:02,
    AUTHOR = {Perugia, I. and Sch\"{o}tzau, D. and Monk, P.},
     TITLE = {Stabilized interior penalty methods for the time-harmonic
              {M}axwell equations},
   JOURNAL = {Comput. Methods Appl. Mech. Engrg.},
  FJOURNAL = {Computer Methods in Applied Mechanics and Engineering},
    VOLUME = {191},
      YEAR = {2002},
    NUMBER = {41-42},
     PAGES = {4675--4697},
       optDOI = {10.1016/S0045-7825(02)00399-7},
       optURL = {https://doi.org/10.1016/S0045-7825(02)00399-7},
}

@article {HoPeS:07,
    AUTHOR = {Houston, P. and Perugia, I. and Sch\"otzau, D.},
     TITLE = {An a posteriori error indicator for discontinuous {G}alerkin
              discretizations of {$H$}(curl)-elliptic partial differential
              equations},
   JOURNAL = {IMA J. Numer. Anal.},
  FJOURNAL = {IMA Journal of Numerical Analysis},
    VOLUME = {27},
      YEAR = {2007},
    NUMBER = {1},
     PAGES = {122--150},
       optDOI = {10.1093/imanum/drl012},
       optURL = {https://doi.org/10.1093/imanum/drl012},
}

@article {VeeserVerfurth12,
    AUTHOR = {Veeser, A. and Verf\"urth, R.},
     TITLE = {Poincar\'e{} constants for finite element stars},
   JOURNAL = {IMA J. Numer. Anal.},
  FJOURNAL = {IMA Journal of Numerical Analysis},
    VOLUME = {32},
      YEAR = {2012},
    NUMBER = {1},
     PAGES = {30--47},
      ISSN = {0272-4979,1464-3642},
   MRCLASS = {65N30 (26D10 46E35 65N15)},
  MRNUMBER = {2875242},
MRREVIEWER = {Alexei\ Bespalov},
       DOI = {10.1093/imanum/drr011},
       URL = {https://doi.org/10.1093/imanum/drr011},
}
\end{document}